\documentclass[12pt,reqno]{amsart}

\usepackage{amsthm}
\usepackage{color}
\usepackage{amssymb, amsmath, amsfonts, latexsym}
\usepackage{enumerate}
\usepackage{booktabs}
\usepackage[table,xcdraw]{xcolor}

\usepackage{bbm}%  \mathbb
\usepackage{hyperref} % \href{URL}{text}

\newtheorem{thm}{Theorem}[]
\newtheorem{lem}{Lemma}[]

\theoremstyle{remark}
\newtheorem{rem}[lem]{Remark}
\theoremstyle{definition}

\numberwithin{equation}{section}

\usepackage{graphicx}
\usepackage{subcaption}
\usepackage{float} % ensure figures appear right  location >.<
\graphicspath{ {./images/} }
\usepackage{algorithm,algpseudocode}
\usepackage{lipsum}

\makeatletter
\newenvironment{breakablealgorithm}
{% \begin{breakablealgorithm}
	\refstepcounter{algorithm}% New algorithm
	\hrule height 0.8pt depth0pt \kern2pt% \@fs@pre for \@fs@ruled
	\renewcommand{\caption}[2][\relax]{% Make a new \caption
		{\raggedright\textbf{\ALG@name~\thealgorithm} ##2\par}%
		\ifx\relax##1\relax % #1 is \relax
		\addcontentsline{loa}{algorithm}{\protect\numberline{\thealgorithm}##2}%
		\else % #1 is not \relax
		\addcontentsline{loa}{algorithm}{\protect\numberline{\thealgorithm}##1}%
		\fi
		\kern2pt\hrule  \kern2pt
	}
}{
	\kern2pt\hrule\relax% \@fs@post for \@fs@ruled
}
\makeatother
\author{Yutao Ma}
\title[Rare-event analysis of $\beta$-Jacobi ensemble]{\bf Asymptotically efficient estimators for tail probabilities of extremals of $\beta$-Jacobi ensembles}  

\date{}

\address{Yutao MA\\ School of Mathematical Sciences $\&$ Laboratory  of Mathematics and Complex Systems of Ministry of Education, Beijing Normal University, 100875 Beijing, China.} 
\thanks{The research of Yutao Ma was supported in part by NSFC 12171038, 11871008 and 985 Projects.}
\email{mayt@bnu.edu.cn}

\author{Siyu Wang*}
\address{Siyu Wang\\ School of Mathematical Sciences $\&$ Laboratory of Mathematics and Complex Systems of Ministry of Education, Beijing Normal University, 100875 Beijing, China.}
\email{wang\_siyu@mail.bnu.edu.cn}

\begin{document}

\maketitle
\begin{abstract}
In this paper, we consider the tail probabilities of extremals of $\beta$-Jacobi ensemble which plays an important role in multivariate analysis.  The key steps in constructing estimators rely on the rate functions of large deviations. Therefore, under specific conditions, we consider stretching and shifting transformations applied to the $\beta$-Jacobi ensemble to ensure that its extremals satisfy the large deviations. The estimator we construct characterize the large deviation behavior and moderate deviation behavior of extremals under different assumptions.
\end{abstract} 


{\bf AMS Classification Subjects 2020:} 60F10, 	60B20

\section{Introduction}
Estimating the probability of rare events or fluctuations in random systems is a crucial challenge encountered in various applied fields, such as engineering \cite{Shwartz}, where a rare event might signify a design failure, or chemistry, where transitions between chemical species or polymer states arise from rare events in a free energy landscape \cite{Wales}.

Large deviation theory serves as a valuable tool for quantifying the probabilities of rare events. However, its practical applications encounter challenges, as approximations based on large deviation results may lack direct applicability, leading to insufficient accuracy in estimating the probabilities of rare events.

Importance sampling emerges as one of the most popular approaches in rare event analysis. While traditional Monte Carlo methods often require substantial time to simulate rare events, importance sampling techniques effectively reduce the simulation time, enhancing overall efficiency. Essentially, importance sampling, a modified Monte Carlo sampling, produces estimates with a smaller coefficient of variation-a natural performance metric for rare event estimation, easily computed through equation \eqref{errorvar} below. Consequently, constructing effective importance sampling estimators for rare events becomes a pivotal task.

\subsection{Basic importance sampling}
Consider a random variable $Y$ with probability distribution $Q$ taking values on spaces $\Omega$, and a real-valued function $\omega$ on $\Omega$. 
Importance sampling (IS) is based on the basic identity, for fixed $B \in \mathcal B,$
$$
	 \mathbb E\left[\mathbf 1_{\left\{\omega\left(Y\right) \in B \right\}}\right]
	 = \mathbb E^R\left[\mathbf 1_{\left\{\omega\left(Y\right) \in B \right\}} \frac{\mathrm{d} Q}{\mathrm{d} R}\right],  
$$
where $R$ is a probability measure such that the Radon-Nikodym derivative $\mathrm{d} Q / \mathrm{d} R$ is well defined on the set $\left\{\omega\left(Y\right) \in B \right\}$. 
Here, $\mathcal B$ denotes the Borel $\sigma$-algebra on $\mathbb R$, and $\mathbb E$ or $\mathbb E^R$ denotes the expectations under the measures $Q$ or $R$, respectively. Then, the random variable $F$ is given by 
\begin{align}\label{deff} 
	F=\frac{\mathrm{d} Q}{\mathrm{d} R} \mathbf 1_{\left\{\omega\left(Y\right) \in B\right\}} 
\end{align} 
is an unbiased estimator of $\mathbb P\left( \omega\left(Y\right) \in B \right)$ under the measure $R$. An averaged importance sampling estimator based on the measure $Q$ is obtained by simulating $N$ i.i.d. copies $F^{(1)}, \ldots, F^{(N)}$ of $F$ under $R$ and computing the empirical average $\widehat{F}_N=\left(F^{(1)}+\cdots+F^{(N)}\right) / N$. 
Moreover, its variance is
$$
\mathrm{Var}^R\left( \widehat{F}_N \right) = \frac{\mathbb E^R \left[F^2\right] - \mathbb P\left(\omega\left(Y\right) \in B \right)^2 }{N},
$$
which decreases as $N$ increases. The coefficient of variation is given by
\begin{align}\label{errorvar}
	 \mathrm{C. O. V.} \left(\widehat{F}_N\right) :=	\frac{\sqrt{\mathrm{Var}^R\left( \widehat{F}_N \right)}}{\mathbb E^R\left( \widehat{F}_N \right) }  = \frac{1}{\sqrt{N}} \left(\frac{\mathbb E^R \left[F^2\right]}{\mathbb P\left(\omega\left(Y\right) \in B \right)^2}-1\right)^{1/2}.
\end{align}  
A central issue is to optimize the left-hand side of  \eqref{errorvar} via selecting $R.$  
If we take $R$ as $$R^*(\cdot):=\mathbb P\left(\cdot \mid \omega\left(Y\right) \in B\right),$$ it is easy to verify that $${\rm Var}^{R^*}(\widehat{F}_N)=0.$$ Unfortunately, the optimal IS density $R^*$ is computationally unavailable since we would need to know the value of $\mathbb P\left( \omega\left(Y\right) \in B \right).$

Given a family of sets $(B_n; n\geq 1)$ such that $$0<\mathbb P\left( \omega\left(Y\right) \in B_n \right) \text{ and } \lim_{n\to\infty}\mathbb{P}\left( \omega\left(Y\right) \in B_n \right)=0,$$ it is necessary to consider the quotient $$\frac{\mathbb{E}^R \left[F_n^2\right]}{\mathbb P\left(\omega\left(Y\right) \in B \right)^2}$$ when constructing a sequence of estimators $(F_n; n\geq 1)$ for $\mathbb P\left( \omega\left(Y\right) \in B_n \right)$.  In the context of rare-event simulations (e.g., Asmussen and Glynn, \cite{Asmussen}), they say that $F_n$ is strongly efficient if
\begin{align*}%\label{defstrongeff}
		\varlimsup _{n \rightarrow \infty} \frac{\mathbb E^R \left[F_n^2\right]}{\mathbb P\left(\omega\left(Y\right) \in B_n\right)^2}<\infty .
\end{align*}  

Another efficiency concept slightly weaker than strongly efficient is asymptotically (or logarithmically) efficient: $\mathbb E^R \left[F_n^2\right]$ tends to $0$ so quickly that
$$
	\varlimsup _{n \rightarrow \infty} 
	\frac{\mathbb E^R \left[F_n^2\right]}{\mathbb P\left(\omega\left(Y\right) \in B_n\right)^{2- \varepsilon}}
	 = 0
$$
for all $\varepsilon>0$, or, equivalently, that
$$
	\lim_{n \rightarrow \infty} \frac{\left| \log \mathbb E^R \left[F_n^2\right]\right|}{2 \left|\log \mathbb P\left(\omega\left(Y\right) \in B_n \right)\right|} = 1.
$$
%The reasons for working on logarithmic efficiency rather than strong efficiency are the following: the difference between them is minor from a practical point of view; logarithmically efficient estimators exist for some main examples, whereas strongly efficient estimators do not (or, at least, have not yet been discovered); and logarithmic efficiency is often much easier than strong efficiency to be verified. 
See Asmussen and Glynn \cite{Asmussen}, Bucklew \cite{Bucklew} or Juneja and Shahabuddin \cite{Juneja} for discussions on efficiency in rare-event simulations.

%%%%%%%%%%%%%%%%%%%%%%%%%%%%%%%%%%%%%
%%%%%%%%%%%%%%%%%%%%%%%%%%%%%%%%%%%%%
\subsection{Motivations}
Asmussen and Kroese (\cite{Asmussen2006}) addressed the problem of tail probability estimation for independent and identically distributed heavy-tailed random variables, developing importance sampling estimators through a two-step construction.  Following a similar structural approach, Jiang {\it et al.} (\cite{JiangRare}) delved into the asymptotical behavior of top eigenvalues for the $\beta$-Laguerre ensemble, whose density is proportional to
$$
\prod_{1 \leq i<j \leq n}\left|x_i-x_j\right|^\beta \cdot \prod_{i=1}^n x_i^{\frac{\beta(p_1-n+1)}{2}-1} \cdot e^{-\frac{1}{2} \sum_{i=1}^n x_i}
$$
with $\beta>0$ and $p_1 \geq n$.  Let $\left(\mu_1, \mu_2, \cdots, \mu_n\right)$ be the $\beta$-Laguerre ensemble. 
		Given $\beta>0$ and assuming $n/p_1 \to \gamma \in (0, 1]$, Jiang {\it et al.} (\cite{JiangRare}) have developed asymptotically efficient estimator for the tail probability $$
	\mathbb P\left( \max _{1 \leq i \leq n} \mu_i > p_1 x \right), \quad x > \beta \left(1+\sqrt{\gamma}\right)^2.
	$$  
	For the case when $\gamma = 0$, they constructed an asymptotically effective estimator by employing the "three-step peeling" technique for precise estimation of the tail probability.
Under more stringent conditions $n^5 / p_1^3 = o(1)$ and $n^5 / p_1^3 = O(1)$, the tail probability has received even more accurate estimations, leading to a strongly efficient estimator.

Motivated by \cite{JiangRare}, the main goal of this paper is to construct asymptotically effective estimators for the tail probabilities of the $\beta$-Jacobi ensemble.  
The study of the $\beta$-Jacobi ensemble finds applications in various domains. In statistics, this ensemble emerges in the context of multivariate analysis of variance, particularly from a pair of independent Gaussian matrices $N_1, N_2$, see \cite{Muirhead}.
% If $N_1$ and $N_2$ have independent real entries, their generalized singular values (squared) follow the law of the Jacobi ensemble with $\beta=1$. When $N_1$ and $N_2$ have independent complex entries, their generalized singular values (squared) follow the law of the Jacobi ensemble with $\beta=2$. 
Moreover, in Statistical Mechanics, the Jacobi ensemble serves as a model for a log-gas confined to the interval $[0,1]$, see \cite{Dyson62}.  

Studying the tail probabilities of the $\beta$-Jacobi ensemble contributes to a deeper understanding of system stability, phase transitions, and critical phenomena. 
The construction of asymptotically efficient estimators makes it easier to compute tail probabilities for extreme values in complex systems, thereby enhancing both the feasibility and accuracy of numerical methods.
In the context of large samples, asymptotically effective estimators typically have lower computational costs compared to strongly efficient ones. Therefore, in situations where computational resources are limited, asymptotically efficient estimators are appropriate options.

A $\beta$-Jacobi ensemble $\mathcal J_n(p_1, p_2) $ is a set of random variables $\lambda:=\left(\lambda_1, \lambda_2, \cdots, \lambda_n\right) \in[0,1]^n$ with joint probability density function
\begin{align}\label{defjacobi} f_n^{p_1, p_2}\left(x_1, \cdots, x_n\right)=C_n^{p_1, p_2} \prod_{1 \leq i<j \leq n}\left|x_i-x_j\right|^\beta \prod_{i=1}^n x_i^{r_{1, n}-1}\left(1-x_i\right)^{r_{2, n}-1},
 \end{align}
where $p_1, p_2 \geq n, r_{i, n}:=\frac{\beta\left(p_i-n+1\right)}{2}$ for $i=1,2$ and $\beta>0$, and the normalizing constant $C_n^{p_1, p_2}$ is given by
$$
C_n^{p_1, p_2}=\prod_{j=1}^n \frac{\Gamma\left(1+\frac{\beta}{2}\right) \Gamma\left(\frac{\beta(p-n+j}{2}\right)}{\Gamma\left(1+\frac{\beta j}{2}\right) \Gamma\left(\frac{\beta\left(p_1-n+j\right)}{2}\right) \Gamma\left(\frac{\beta\left(p_2-n+j\right)}{2}\right)}
$$
with $p=p_1+p_2$ and $\Gamma(a)=\int_0^{\infty} x^{a-1} e^{-x} dx$ with $a>0.$

%For the case $0<\sigma \gamma \leq 1$, the first author in \cite{Ma} established the large deviation for the empirical measure $\frac{1}{n} \sum_{i=1}^n \delta_{\lambda_i}$ under the assumption ${\bf A}$. Consequently, it follows that $$\lambda_{(n)} \longrightarrow \frac{\sigma}{1+\sigma}\left(\sqrt{1-\frac{\sigma \gamma}{1+\sigma}}+\sqrt{\frac{\gamma}{1+\sigma}}\right)^2=:u_2$$ almost surely as $n \rightarrow \infty$.

Jiang {\it et al.} in \cite{JiangRare}  used the tridiagonal random matrix characterization of the $\beta$-Laguerre ensemble and the rate function of the large deviation of $\max _{1 \leq i \leq n} \mu_i /n$ to construct an asymptotically efficient estimator for the tail probability $\mathbb P\left( \max _{1 \leq i \leq n} \mu_i > p_1 x \right).$ As a parallel model of $\beta$-Laguerre ensemble, Killip and Nenciu (\cite{Killip}) constructed a tridiagonal random matrix for the $\beta$-Jacobi ensemble, which makes it possible for us to follow the method in \cite{JiangRare} to construct an asymptotically efficient estimator for tail probabilities related to top eigenvalues once the corresponding large deviations are established. 

We introduce first the common conditions utilized in this paper:
	$$
	{\bf A} : \quad \lim _{n \rightarrow \infty} \frac{p_1}{p_2}=\sigma \geq 0, \quad \lim _{n \rightarrow \infty} \frac{n}{p_1}=\gamma \in\left[0,1 \wedge \frac{1}{\sigma}\right), \text { and } \quad \lim _{n \rightarrow \infty} \frac{\log n}{\beta n}=0.
	$$
	
		Let $\left(\lambda_1, \lambda_2, \cdots, \lambda_n\right)$ be the $\beta$-Jacobi ensemble $\mathcal J_n(p_1, p_2)$ with joint density function \eqref{defjacobi}.  Let  $\lambda_{(1)}< \cdots <\lambda_{(n)}$ be the order statistics of  $\lambda_1, \cdots, \lambda_n$.
In cases where $\gamma \sigma = 0$, the top eigenvalue $\lambda_{(n)}$ does not satisfy the  large deviation. In such scenarios, to estimate the tail probability $\mathbb P \left( \lambda_{(n)} > x \right)$, we need to consider the following random variables
$$
X_i := \frac{\left(p_1+p_2\right) \lambda_{i}-p_1}{\sqrt{ n p_1}}
$$
for $1 \leq i \leq n,$ which were introduced in \cite{Ma} for the large deviation of the empirical measures.   

Precisely, in \cite{Ma}, the first author established the large deviation for $\frac1n\sum_{i=1}^n \delta_{X_i}$ under the condition {\bf A} and 
 then proved that
\begin{align}
	X_{(n)} & \longrightarrow \frac{(1-\sigma) \sqrt{\gamma}+2 \sqrt{1+\sigma-\sigma \gamma}}{1+\sigma}=:\widetilde{u}_2, \label{defu2} \\
	X_{(1)} & \longrightarrow \frac{(1-\sigma) \sqrt{\gamma}-2 \sqrt{1+\sigma-\sigma \gamma}}{1+\sigma} =:\widetilde{u}_1, \label{defu1}
\end{align}
almost surely as $n\to\infty.$   
From \eqref{defu2} and \eqref{defu1}, it is evident that  
\begin{align*}
	&\mathbb{P} \left(X_{(1)} < x \right) \longrightarrow 0, \;\forall \, x \in \left(-\frac{1}{\sqrt{\gamma}}, \;\widetilde u_1\right), \;
	\mathbb{P} \left(X_{(n)} > x \right) \longrightarrow 0, \;\forall\, x \in \left(\widetilde u_2, \;\frac{1}{\sqrt{\gamma}\sigma}\right).
\end{align*}  
Note that if $\gamma=0$ or $\gamma \sigma=0$ we will substitute $\frac{1}{\sqrt{\gamma}}$ or  $\frac{1}{\sqrt{\gamma} \sigma}$ with $+\infty,$ respectively.

As the construction of asymptotically efficient estimators critically depends on the rate functions of the corresponding large deviations,  we start with the large deviations for $X_{(1)}$ and $X_{(n)}$ and then 
we verify the asymptotic efficiency of the corresponding importance sampling estimator.

\subsection{Large deviations of scaled extremals}
For the statement of the large deviations, we first introduce three well-known distributions in random matrix theory, which appeared as the weak limits of empirical measures of eigenvalues of some random matrices and were involved in the rate functions. Here they are: the semicircular law, the Marchenko-Pastur law, and the Wachter law. 

The density function of the semicircular law $c_\alpha$ for $\alpha>0$ is given by
$$
	c_\alpha(x)=\frac{1}{\pi \alpha} \sqrt{2 \alpha-x^2} \mathbf 1_{\left\{|x| \leq \sqrt{2 \alpha}\right\}}.
$$
The Marchenko-Pastur law $\mu_\gamma$ for $\gamma>0$ is a distribution whose density function $h_\gamma$ is given by
$$
	h_\gamma(x)=\frac{1}{2 \pi \gamma x} \sqrt{\left(x-\gamma_1\right)\left(\gamma_2-x\right)} \mathbf 1_{\left\{\gamma_1 \leq x \leq \gamma_2\right\}}
$$
with $\gamma_1=(\sqrt{\gamma}-1)^2$ and $\gamma_2=(\sqrt{\gamma}+1)^2$. 
The Wachter distribution $\nu_{\gamma, \sigma}$ for $\gamma>0, \sigma>0$ is much more complicated, whose density function is given by
$$
h_{\gamma, \sigma}(x)=\frac{1+\sigma}{2 \pi \sigma \gamma} \frac{\sqrt{\left(x-u_1\right)\left(u_2-x\right)}}{x(1-x)} \mathbf 1_{\left\{u_1 \leq x \leq  u_2\right\}},
$$
where $u_1=\frac{\sigma}{1+\sigma}\left(\sqrt{1-\frac{\sigma \gamma}{1+\sigma}}-\sqrt{\frac{\gamma}{1+\sigma}}\right)^2$ and $u_2=\frac{\sigma}{1+\sigma}\left(\sqrt{1-\frac{\sigma \gamma}{1+\sigma}}+\sqrt{\frac{\gamma}{1+\sigma}}\right)^2$.

As a consequence of large deviation of the empirical measure $\mu_n:=\frac1n\sum_{i=1}^n\delta_{X_i},$  the first author in \cite{Ma} showed that $\mu_n$ converges weakly to $\widetilde{\nu}_{\gamma, \sigma}$ with probability one, whose density function is 
\begin{align}\label{dennu}
	\widetilde{h}_{\gamma, \sigma}(x) 
	= 
	\begin{cases}
		\frac{\sigma \sqrt{\gamma}}{1+\sigma} h_{\gamma, \sigma}\left(\frac{\sigma}{1+\sigma}(1+\sqrt{\gamma} x)\right), & 0<\sigma \gamma < 1; \\ 
		\sqrt{\gamma} h_\gamma(1+\sqrt{\gamma} x), & \sigma=0,0<\gamma \leq 1; \\ 
		\sqrt{1+\sigma}c_2( \sqrt{1+\sigma} x ), & \gamma=0, \sigma \geq 0.
	\end{cases}	
\end{align}

Large sample properties of the extremal eigenvalues have been extensively studied in the literature, most of which focus on the asymptotic distribution and its large deviation principle.
In \cite{Huang} and \cite{Jiang}, they derive the asymptotic distributions and the limits of extremal eigenvalues when $\beta > 0$ is fixed.  

 Under the assumption {\bf A}, Lei and Ma (\cite{MaLD}) have already shown the large deviations for $ p \lambda_{(n)}/p_1 $. 
 It is noteworthy that when $n=O(p_1),$ $(X_i)_{1\le i\le n}$ and $\left(p \lambda_{i}/p_1\right)_{1\le i\le n}$ behave similarly while they are completely different when $n=o(p_1).$  We will further study the large deviation of $X_{(1)}$ and $X_{(n)}$ as follows.

\begin{thm}\label{maxldp}
Let $\left(\lambda_1, \lambda_2, \cdots, \lambda_n\right)$ be the $\beta$-Jacobi ensemble $\mathcal J_n(p_1, p_2)$ with density function \eqref{defjacobi}. Under the assumption {\bf A}, the sequence $\frac{\left(p_1+p_2\right) \lambda_{(n)}-p_1}{\sqrt{np_1 }}$ satisfies a large deviation principle with speed $\beta n$ and good rate function $J_{\gamma, \sigma}$ given by
$$
J_{\gamma, \sigma}(x)  = 
	\begin{cases} 
		- z_{\gamma, \sigma} - \int \log|x-y|  \widetilde{\nu}_{\gamma, \sigma}(\mathrm{d} y) - \varphi_{\gamma, \sigma}(x), & x \in\left[\widetilde{u}_2, \frac{1}{\sqrt{\gamma} \sigma}\right) ; \\ 
		+\infty, & other.
	\end{cases}
$$
Moreover, $\widetilde u_2$ is the unique minimizer of $ J_{\gamma, \sigma}$ and $J_{\gamma, \sigma}(\widetilde u_2)=0.$ \end{thm}

\begin{thm}\label{minldp}
	Let $\left(\lambda_1, \lambda_2, \cdots, \lambda_n\right)$ be the $\beta$-Jacobi ensemble  $\mathcal J_n(p_1, p_2)$  with density function \eqref{defjacobi}. Under the assumption {\bf A}, the sequence $\frac{\left(p_1+p_2\right) \lambda_{(1)}-p_1}{\sqrt{np_1 }}$ satisfies a large deviation principle with speed $\beta n$ and good rate function $I_{\gamma, \sigma}$ given by
$$
I_{\gamma, \sigma}(x) = \begin{cases} - z_{\gamma, \sigma} - \int \log|x-y|  \widetilde{\nu}_{\gamma, \sigma}(\mathrm{d} y) - \varphi_{\gamma, \sigma}(x), & x \in\left(-\frac{1}{\sqrt{\gamma}} , \widetilde{u}_1 \right] ; \\ +\infty, & other.
\end{cases}
$$
Moreover, $ I_{\gamma, \sigma}$ achieves its minimum value $0$ at the unique point  $\widetilde u_1.$ 
\end{thm}
Here
\begin{align}\label{vargs}
	\varphi_{\gamma, \sigma}(x) 
	= 
		\begin{cases}
			\frac{1-\gamma}{2 \gamma} \log (1+\sqrt{\gamma} x)+\frac{1-\gamma \sigma}{2 \gamma \sigma} \log (1-\sqrt{\gamma} \sigma x), & 0<\sigma \gamma < 1;\\ 
			\frac{1-\gamma}{2 \gamma} \log (1+\sqrt{\gamma} x)-\frac{1}{2 \sqrt{\gamma}} x, & \sigma=0,0<\gamma \leq 1; \\
			-\frac{1+\sigma}{4} x^2, & \gamma=0, \sigma \geq0;
		\end{cases} 
\end{align}
and
\begin{align}\label{zgs} 
	z_{\gamma, \sigma} 
	= 
		\begin{cases}
			\frac{1}{2} \log (1+\sigma)+\frac{1+\sigma-\gamma \sigma}{2\gamma \sigma} \log \left(1+\frac{\gamma \sigma}{1+\sigma-\gamma \sigma}\right), & 0<\sigma \gamma < 1; \\ 
			\frac{1}{2}, & \sigma=0,0<\gamma \leq 1; \\ 
			\frac{1}{2} \log (1+\sigma)+\frac{1}{2}, & \gamma=0, \sigma \geq0;
		\end{cases}
\end{align}
and $x$ belongs to its natural domain.

\begin{rem}\label{lx}
	Using the transformation $$\mathbb{P}(X_{(n)}>x)=\mathbb{P}\left(\lambda_{(n)}>  \frac{p_1}{p} +\frac{\sqrt{np_1 }}{p} \; x\right),$$ we can characterize the asymptotical behavior  of  $\lambda_{(n)}$ through the large deviation related to $X_{(n)}.$ Specifically, when $0 < \gamma \sigma < 1,$ $\lambda_{(n)}$ and $X_{(n)}$ behave similarly. In the case of $\gamma \sigma =0,$ the probability $\mathbb{P}\left(X_{(n)}>x\right)$ characterizes the asymptotical behavior of $\lambda_{(n)}$ around the neighborhood of $\frac{\sigma}{1 + \sigma}$, which is related to the moderate deviation of $\lambda_{(n)}$. 
\end{rem}

\subsection{Asymptotical efficiency of the importance sampling estimators} 

Now, we are at the position to state the asymptotical efficiency of the importance sampling estimators. For brevity,  let $\left(\lambda_1, \lambda_2, \cdots, \lambda_n\right)$ and $\left(\widetilde{\lambda}_1, \widetilde{\lambda}_2, \cdots, \widetilde\lambda_{n-1}\right)$ be the $\beta$-Jacobi ensemble $\mathcal J_n(p_1, p_2) $ and $\mathcal J_{n-1}(p_1-1, p_2-1)$, respectively.  %For consistent notation, when constructing estimators for the tail probabilities of the maximum and minimum values, consider $\left(\widetilde{\lambda}_1, \widetilde{\lambda}_2, \cdots, \widetilde\lambda_{n-1}\right)$and  $\left(\widehat{\lambda}_2, \widehat{\lambda}_3, \cdots, \widehat{\lambda}_{n}\right)$ be the $\beta$-Jacobi ensemble $\mathcal J_{n-1}(p_1-1, p_2-1).$

\subsubsection{\bf The IS estimator of $X_{(n)}$ and $X_{(1)}$}
We first state the asymptotical efficiency of 
the importance sampling estimator of $\mathbb P\left(X_{(n)}> x\right)$ for $x \in\left(\widetilde u_2, \frac{1}{\sqrt{\gamma} \sigma}\right).$ 

Set $X_i=\frac{\left(p_1+p_2\right) \lambda_{i}-p_1}{\sqrt{ n p_1}}$ for $1 \leq i \leq n$ and $\widetilde X_i=\frac{\left(p_1+p_2\right) \widetilde \lambda_{i}-p_1}{\sqrt{ n p_1}}$  for $1 \leq i \leq n-1.$ 
Let $\widetilde X_{n}$ be a random variable on $\Omega,$ whose conditional density function given $(\widetilde X_1, \cdots, \widetilde X_{n-1})$ is  \begin{align}\label{expxn}
	h\left(y\right):= \frac{r(x)}{1-e^{-r(x)\left(p_2/\sqrt{np_1}- x \vee \widetilde X_{(n-1)}  \right)}} e^{-r(x)\left(y-x \vee \widetilde X_{(n-1)} \right)} \mathbf 1_{\left\{ x \vee \widetilde X_{(n-1)}<y < p_2/\sqrt{np_1}\right\}},
\end{align} where $J_{\gamma, \sigma}$ is introduced in Theorem \ref{maxldp} and $0<r(x)< 2 \beta (n-1) J_{\gamma, \sigma}^{\prime}(x).$   
Let $Q_n^{p_1, p_2}$ and $R_n^{p_1, p_2}$ be the joint distribution of $\left(X_{(1)}, \cdots, X_{(n)}\right)$ and $\left(\widetilde X_{(1)}, \cdots, \widetilde X_{(n-1)}, \widetilde X_{n} \right)$, respectively. 

According to \eqref{deff}, we define the importance sampling estimator of  $\mathbb P\left(X_{(n)}> x\right)$ as follows:
\begin{align}\label{xndeffn}
F_n(x) = \frac{\mathrm d Q_n^{p_1, p_2}}{\mathrm d R_n^{p_1, p_2}} (X_{(1)}, \cdots, X_{(n)})\mathbf 1_{\left\{X_{(n)}> x\right\}}. 	
\end{align}
 
 When there's no ambiguity, we simplify the notation $F_n(x)$ and $r(x)$ to $F_n$ and $r,$ respectively.

\begin{thm} \label{fnasy}
Under the assumption {\bf A}, for $x \in\left(\widetilde u_2, \frac{1}{\sqrt{\gamma} \sigma}\right)$, the importance sampling estimator $F_n$ of $\mathbb P\left(X_{(n)}> x\right)$ is asymptotically efficient. That is 
$$
\lim _{n \rightarrow \infty} \frac{\log \mathbb{E}^R\left[F_n^2\right] }{2 \log \mathbb P\left(X_{(n)}> x\right)} = 1 .
$$
\end{thm}

\begin{rem}\label{rem2}
Theorem \ref{fnasy} remains valid when $x$ is replaced by $x_n$ once  $\lim\limits_{n \rightarrow \infty} x_n \in \left(\widetilde u_2, \frac{1}{\sqrt{\gamma} \sigma}\right).$ 
\end{rem}

Similarly, via the large deviation principle of $X_{(1)}$ obtained in Theorem \ref{minldp},  we can construct asymptotically efficient importance sampling estimator $G_n$ of the tail probability  $\mathbb P\left(X_{(1)} <  x \right)$ for $x \in \left(-\frac{1}{\sqrt{\gamma}}, \widetilde u_1\right)$.

%Set $\widehat X_i = \frac{(p_1+p_2) \widehat \lambda_i -p_1}{\sqrt{n p_1}} $ for $2 \leq i \leq n. $  
Let $\widehat X_{1}$ be a random variable, whose conditional density function given 
%$\left(\widehat X_2, \cdots, \widehat X_n\right)$ 
$(\widetilde X_1, \cdots, \widetilde X_{n-1})$ is $$\hat h\left(y\right):= \frac{r}{1-e^{-r\left( \sqrt{p_1/n}+ x \wedge \widetilde X_{(1)} \right)}} e^{r\left(y-x \wedge \widetilde X_{(1)} \right)} \mathbf 1_{\left\{- \sqrt{p_1/n} < y < x \wedge \widetilde X_{(1)} \right\}}$$ with $I_{\gamma, \sigma}$ being the rate function related to $X_{(1)}$ and $0<r<2 \beta (n-1) I_{\gamma, \sigma}^{\prime}(x).$  
Let $T_n^{p_1, p_2}$ be the joint distribution of $\left(\widehat{X}_{1}, \widetilde X_{(1)}, \cdots, \widetilde X_{(n-1)} \right).$
\begin{thm}\label{thfnasy} 
Under the assumption {\bf A}, the importance sampling estimate $$G_n  = \frac{\mathrm d Q_n^{p_1, p_2}}{\mathrm d T_n^{p_1, p_2}} (X_{(1)}, \cdots, X_{(n)}) \, \mathbf 1_{\{X_{(1)} < x\} } $$ of $\mathbb P\left(X_{(1)} < x\right)$ satisfies 
$$
\lim _{n \rightarrow \infty} \frac{\log \mathbb{E}^R\left[G_n^2\right] }{2 \log \mathbb P\left(X_{(1)}< x\right)} = 1 .
$$ 
\end{thm}
\subsubsection{\bf The IS estimator of $\lambda_{(n)}$ and $\lambda_{(1)}$}
Next, we present results for $\lambda_{(n)}$. 
Using the transformation
$$
\mathbb P \left( \lambda_{(n)} > x \right) = \mathbb P\left( X_{(n)} >  \frac{px -p_1}{\sqrt{np_1}} \right) 
$$
and taking account of Remark \ref{rem2},
we can characterize $\mathbb P \left( \lambda_{(n)} > x \right)$ by the importance sampling estimator $ F_{n}\left(\frac{px -p_1}{\sqrt{np_1}}  \right) $  when $0<\gamma\sigma<1$ since $x_n:=\frac{px -p_1}{\sqrt{np_1}}$  satisfies $$\lim\limits_{n\to\infty} x_n=\frac{(1+\sigma)x-\sigma}{\sqrt{\gamma}\sigma}\in\left(\widetilde u_2, \frac{1}{\sqrt{\gamma} \sigma}\right).$$ 
Here is the statement. 
\begin{thm}\label{thmln}
	Under the assumption {\bf A}, when $0<\gamma\sigma<1$ and for $x \in (u_2,1)$, the importance sampling estimator $ F_{n}\left(\frac{px -p_1}{\sqrt{np_1}}  \right) $ is asymptotically efficient for $\mathbb P\left(\lambda_{(n)}> x\right).$ In the case of $\gamma \sigma =0,$ let $ x = u_2 + \frac{\sqrt{np_1}}{p}y$ for $y>0.$ Thus the importance sampling estimator $F_n\left( \frac{pu_2 -p_1}{\sqrt{np_1}}  + y \right)$ is asymptotically efficient for $\mathbb P\left(\lambda_{(n)}> u_2 + \frac{\sqrt{np_1}}{p}y \right),$ which is related to the moderate deviation of $\lambda_{(n)}.$
\end{thm}

\paragraph{\bf The IS estimator of $Z_{(1)}$ and $Z_{(n)}$} Let's define $Z_i = \frac{\left(p_1+p_2\right) \lambda_{i}}{p_1},$ for $1 \leq i \leq n.$ In \cite{Ma}, Ma proves that
\begin{align*}
	Z_{(n)} \longrightarrow \left(\sqrt{1-\frac{\sigma \gamma}{1+\sigma}}+\sqrt{\frac{\gamma}{1+\sigma}}\right)^2  \quad \text{ and } \quad
	Z_{(1)} \longrightarrow \left(\sqrt{1-\frac{\sigma \gamma}{1+\sigma}}-\sqrt{\frac{\gamma}{1+\sigma}}\right)^2 
\end{align*}
almost surely as $n\to\infty.$   
From
$$\mathbb{P}(X_{(n)}>x)=\mathbb{P}\left(Z_{(n)}> 1+\sqrt{\frac{n}{p_1} } \; x\right),$$ we observe that $Z_{(n)}$ and $Z_{(1)}$ can have similar results and we don not give the statements here.   

 The remainder of this paper is organized as follows: in the next section, we will focus on the simulation of importance sampling. The third section provides the proofs of large deviations, the fourth section is devoted to the proofs of Theorem \ref{fnasy} and \ref{thfnasy}, and we collect some lemmas in the last section.

\section{Simulations of Importance sampling}
In this section, we numerically study the tail probability of the $\beta$-Jacobi ensemble using the IS estimator 
$$
F_n(x) = \frac{\mathrm d Q_n^{p_1, p_2}}{\mathrm d R_n^{p_1, p_2}} (X_{(1)}, \cdots, X_{(n)})\mathbf 1_{\left\{X_{(n)}> x\right\}}.
$$ as defined in \eqref{xndeffn}.
Note that the estimator $F_n(x)$ can be written as
\begin{align}\label{deff1} 
	F_n(x) & = \frac{g_{n}^{p_1, p_2}\left(X_{(1)}, \cdots, X_{(n)}\right)}{g_{n-1}^{p_1-1, p_2-1}\left(X_{(1)}, \cdots, X_{(n-1)}\right) h(X_{(n)})}\mathbf 1_{\left\{X_{(n)}> x\right\}}\\
	& = B_n h^{-1}(X_{(n)}) \prod_{i=1}^{n-1}(X_{(n)} -X_{(i)})^{\beta} u_n (X_{(n)}) \mathbf 1_{\left\{X_{(n)} > x \vee X_{(n-1)} \right\}} , \nonumber  
 \end{align}
where $B_n$ is the normalizing constant and will be introduced later.  

Following the method in \cite{JiangRare}, we characterize the measure $R_n^{p_1,p_2}$ in Algorithm \ref{algoL}. We first sample from the $\beta$-Jacobi ensemble by utilizing the tridiagonal random matrix constructed by Killip and Nenciu \cite{Killip} from the $\beta$-Jacobi ensemble $\mathcal J_n(p_1, p_2)$.

	\vspace{1em}
		\begin{breakablealgorithm}
		\caption{Sampling from the $\beta$-Jacobi ensemble $\mathcal{J}_n(p_1, p_2)$}

		\begin{algorithmic}[1]
			\Require The parameters $\beta, n, p_1, p_2$ in the $\beta$-Jacobi ensemble $\mathcal{J}_n(p_1, p_2)$
			\Ensure $n$-dimensional random vector from $\mathcal{J}_n(p_1, p_2)$		
			\Function{FunJ}{$\beta, n, p_1, p_2$}
			\State{$ \mathcal{J}_n^{p_1, p_2} \gets $ Zero Matrix of $n \times n$}
			\State{$(c(0), s(0)) \gets (0, 0)$}
			
			\For{ $k$ from $1$ to $n$ }
			\State{ $c(k) \gets$ sample from $\operatorname{Beta}\left( \frac{\beta}{2} \left(p_1-k+1  \right), \frac{\beta}{2} \left(p_2-k+1 \right) \right)$}
			\State{ $s(k) \gets $ sample from $\operatorname{Beta}\left( \frac{\beta}{2} \left(n-k \right),  \frac{\beta}{2} \left(p_1+p_2-n-k+1  \right)  \right)$ }
			\State{ $\mathcal{J}_n^{p_1, p_2}(k, k) \gets s(k-1)(1-c(k-1)) +c(k)(1-s(k-1)) $  }
			\While{$k \neq n$} 
			\State{ $ \mathcal{J}_n^{p_1, p_2}(k, k+1) \gets \sqrt{c(k)(1-c(k))s(k)(1-s(k-1))}$}
			\State{ $\mathcal{J}_n^{p_1, p_2}(k+1,k ) \gets \mathcal{J}_n^{p_1, p_2}(k, k+1) $ }
			\EndWhile
			\EndFor
			\State{\textbf{return}  the eigenvalues of $\mathcal{J}_n^{p_1, p_2}$}	
			\EndFunction			
		\end{algorithmic}
	\end{breakablealgorithm}
	
	\vspace{1em}

	\vspace{1em}
	
	\begin{breakablealgorithm}
		\caption{Sampling under the measure $R_n^{p_1,p_2}$ }\label{algoL}
		\begin{algorithmic}[1]
			\Require the parameters $\beta, n, p_1, p_2, x_n$ in $ \mathbb{P}\left(X_{(n)}>  x \right) $
			
			\Ensure $n$ dimensional random vector following the measure $R_n^{p_1,p_2}$ 
				
			\State{ $\left( \lambda_{(1)} , \ldots, \lambda_{(n-1)} \right)  \gets $ the order statistics of FunJ$(\beta, n-1, p_1-1, p_2-1)$ eigenvalues}
			\State{ $\left( \widetilde X_{(1)}, \cdots, \widetilde X_{(n-1)}  \right)  \gets \left( \frac{p \lambda_{(1)} -p_1}{\sqrt{n p_1}}  , \ldots,  \frac{p \lambda_{(n-1)} -p_1}{\sqrt{n p_1}} \right)$ }			
			\If{ $n/p_1 > 0.01$ }
				\State{ $\gamma \gets \frac{n}{p_1} $ }
			\EndIf{ $\gamma \gets 0$ }
		
			\If{ $p_1/p_2 > 0.01$ }
				\State{ $\sigma \gets \frac{p_1}{p_2} $ }
			\EndIf{ $\sigma \gets 0$ }			
			\Repeat 
			\State{$\widetilde X_{n} \gets$ sample from $\operatorname{Exp}\left(\beta (n-1) J_{\gamma, \sigma}^{\prime}(x) \right)$}	
			\State{$\widetilde X_{n} \gets \widetilde X_{n}  + x \vee \widetilde X_{(n-1)} $}
			\Until{($\widetilde X_{n} <  p_2/\sqrt{np_1} $)}
			
			\State{\textbf{return} $\left(\widetilde X_{(1)}, \cdots, \widetilde X_{(n-1)}, \widetilde X_{n} \right)$}
		\end{algorithmic}
	\end{breakablealgorithm}
	
	\vspace{1em}

In lines 9-12 of the above algorithm, given $(\widetilde X_1, \cdots, \widetilde X_{n-1})$, we sample from a truncated exponential distribution with parameter $\beta (n-1) J_{\gamma, \sigma}^{\prime}(x) $ and truncation interval $(x \vee \widetilde X_{(n-1)}, p_2/\sqrt{np_1})$. Then the density function of $\widetilde X_{n}$ is defined as in \eqref{expxn}. With the help of Lemma \ref{stieltjes}, we can easily get 
$$
    J_{\gamma, \sigma}^{\prime}(x) = \frac{\sqrt{\left(\widetilde{u}_2-x\right)\left(\widetilde{u}_1-x\right)}}{(1+\sqrt{\gamma} x)(1-\sigma \sqrt{\gamma} x)} \left(\gamma \sigma \mathbf 1_{\left\{ 0<\sigma \gamma < 1 \right\}} + \frac{1+\sigma}{2}\mathbf 1_{\left\{ \sigma \gamma =0 \right\}} \right).
$$
Here, $ \widetilde{u}_2 $ and $\widetilde{u}_1$ are defined in \eqref{defu2} and \eqref{defu1}, respectively.

Additionally, we showed that the IS estimator $ F_{n}\left(\frac{px -p_1}{\sqrt{np_1}}  \right) $ accurately estimates $\mathbb{P}(\lambda_{(n)}>x)$ with asymptotically efficiency in Theorem \ref{thmln}. 
Recall the empirical average $ \widehat{F}_N=\left(F^{(1)}+\cdots+F^{(N)}\right)/N $ and its coefficient of variation is given by \eqref{errorvar}.
In the following algorithm, we evaluate the actual performance of the empirical average $ \widehat{F}_N\left(\frac{px -p_1}{\sqrt{np_1}}  \right) $ for $ \mathbb{P}\left(\lambda_{(n)}>  x\right) $ by calculating the mean and coefficient of variation of $N$ independent samples. 

\vspace{1em}
\begin{breakablealgorithm}
\caption{Some statistical indicators of $ \widehat{F}_N$}\label{algoF}
\begin{algorithmic}[1]
	\Require the parameters $\beta, n, p_1, p_2, x$ in $ \mathbb{P}\left(\lambda_{(n)}>  x \right) $ and number of simulations $N$.
 	\Ensure The mean Est. and coefficient of variation C.O.V. of $N$ independent samples $F^{(1)}, \ldots, F^{(N)}.$
	\State{ $x \gets \frac{px -p_1}{\sqrt{np_1}}  $ }
	\For{ $k$ from $1$ to $n$ }
		\State{ $\left(\ell_1, \cdots, \ell_n\right) \gets $ sample from measure $R_n^{p_1,p_2}$}	
		\State{ $F^{(k)} \gets B_n \prod_{i=1}^{n-1} \left(\ell_n - \ell_i \right)^{\beta} u_n(\ell_n) h^{-1}\left(\ell_n \right)$}
	\EndFor
	\State{ Est.  $ \gets \sum_{ k=1}^nF^{(k)}/N $ }
	\State{ Std.  $ \gets \sqrt{\sum_{ k=1}^n\left(F^{(k)} - Est. \right)^2/N} $ }
	\State{ C.O.V.  $ \gets$  Std./Est.  }
	\State{\textbf{return} (Est., C.O.V. ) }
\end{algorithmic}
\end{breakablealgorithm}
\vspace{1em}

Utilizing the aforementioned algorithm and the R language, we conducted a series of simulation experiments. Firstly, in Figure (A), we show the relationship between the estimated values of $\mathbb{P}\left(\lambda_{(n)}> x \right)$ obtained from 5000 simulations and $x$ on a logarithmic scale axis. Secondly, in Figure (B), we present the relationship between the coefficient of variation of sampled data $F^{(1)}, \ldots, F^{(N)}$ and the number of simulations $N$ for different values of $x$. For a detailed implementation of the R code, please visit \href{}{https://github.com/S1yuW/Rare-event.}
\begin{figure}[H]\label{covfn}
  \begin{subfigure}{0.24\textwidth}
    \includegraphics[width=\linewidth]{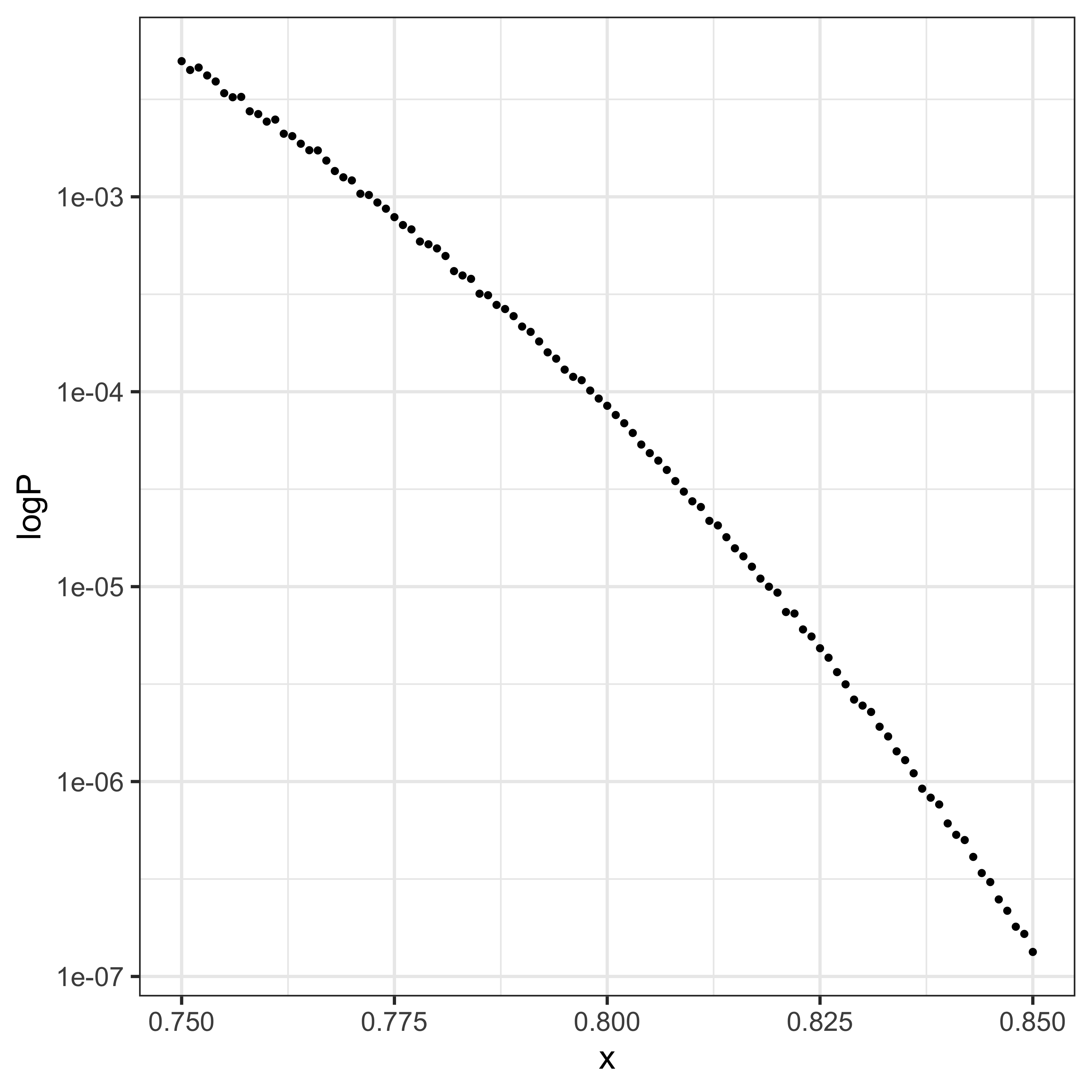}
    \caption{The estimates.}
  \end{subfigure}
  \hfill
  \begin{subfigure}{0.75\textwidth}
    \includegraphics[width=\linewidth]{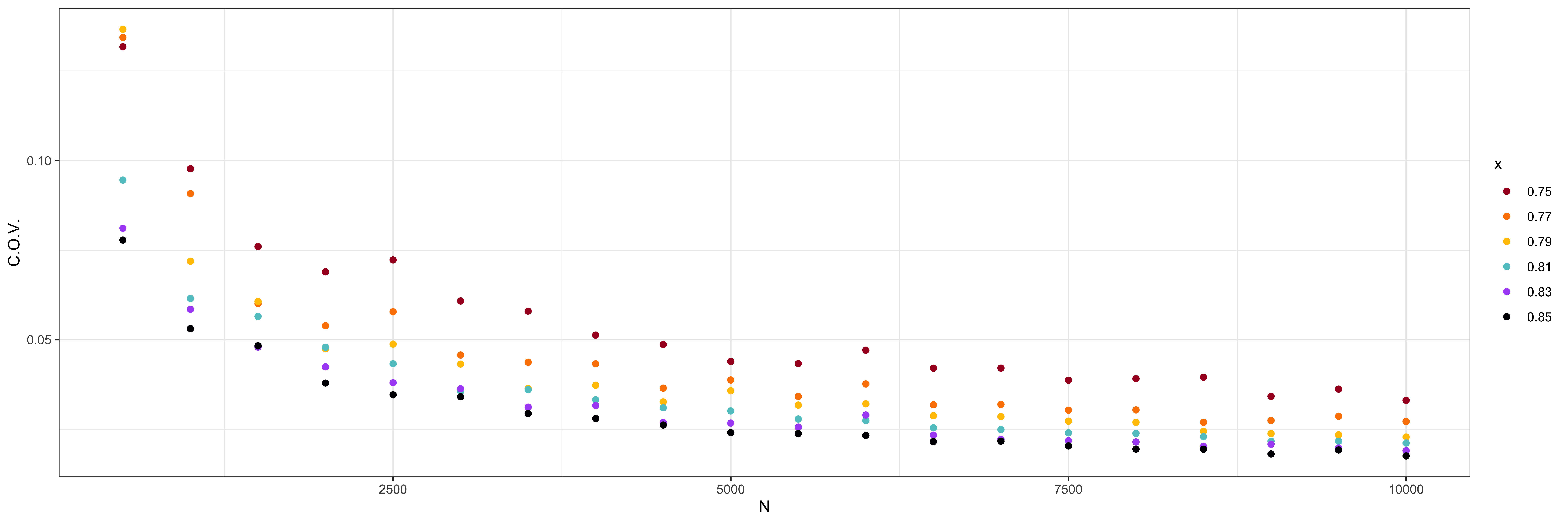}
    \caption{The C.O.V. of $F^{(1)}, \ldots, F^{(N)}$.}
  \end{subfigure}
  \caption{$\beta =2$, $n=10$, $p_1= 20$, and $p_2= 40.$}
\end{figure}
%%%%%%
As the simulation size $N$ increases, the coefficient of variation remains relatively small and stabilizes, indicating the effectiveness of the estimator. Furthermore, while holding other parameters constant, the coefficient of variation decreases as $x$ increases, suggesting that the estimator is more effective for larger $x$ values. 

\section{ Proof of Theorem \ref{maxldp} and Theorem \ref{minldp}}\label{ldp}

In this section, we prove the large deviations of $X_{(n)}$ and $X_{(1)}.$ We begin by presenting the expression for the joint probability distribution $Q_n^{p_1, p_2}$ of $\left(X_{(1)}, \ldots, X_{(n)}\right)$. 
In the subsequent analysis, we introduce the parameters $s_{i, n}=\sqrt{n p_1} / p_i$ and $r_{i, n} = \beta(p_i-n+1)/2$, where $i = 1, 2$. Additionally, we define the function
\begin{equation}\label{defu}
u_n(x) = \left(1+s_{1, n} x\right)^{r_{1, n}-1}\left(1-s_{2, n} x\right)^{r_{2, n}-1}.
\end{equation}
The density function of $Q_n^{p_1, p_2}$ is then given by
\begin{align}\label{orderpdf}
	g_n^{p_1, p_2}\left(x_1, \cdots, x_n\right) 
	=& \bar{C}_{n}^{p_1, p_2} \prod_{1 \leq i<j \leq n}\left(x_j-x_i\right)^\beta \prod_{1 \leq i \leq n} u_n(x_i)  \mathbf 1_{\left\{x_1<\cdots<x_n\right\}}
\end{align}
with the normalizing constant
\begin{align*}
\bar{C}_{n}^{p_1, p_2}
& = n!C_{n}^{p_1, p_2}\left(\frac{s_{1,n} }{s_{1,n}+s_{2,n}} \right)^{\beta \frac{n(n-1)}{2} + n r_{2, n}}\left(\frac{s_{2,n} }{s_{1,n}+s_{2,n}}\right)^{\beta \frac{n(n-1)}{2} + n r_{1, n}}.
\end{align*}
We frequently use the following factorization as  
\begin{align}\label{orderpdf1}
	\mathrm{d}Q_n^{p_1, p_2}
	=  B_n \prod_{i=1}^{n-1}(x_n-x_i)^{\beta} u_n(x_n) \mathbf 1_{\left\{x_n> x_{n-1}\right\}} \mathrm{~d}x_n \mathrm{~d}Q_{n-1}^{p_1-1, p_2-1},
\end{align}
where $B_n$ is defined by
\begin{align}
	B_n =
	n \left(\frac{s_{1,n} }{s_{1,n}+s_{2,n}} \right)^{\beta (n-1) +  r_{2, n}}\left(\frac{s_{2,n} }{s_{1,n}+s_{2,n}}\right)^{\beta (n-1) + r_{1, n}} \frac{C_n^{p_1, p_2}}{C_{n-1}^{p_1-1, p_2-1}}\label{defBn}
\end{align}
and $Q_{n-1}^{p_1-1, p_2-1}$ represents the joint distribution of $\left(\frac{\left(p_1+p_2\right) \widetilde \lambda_{(1)}-p_1}{\sqrt{n p_1}},\ldots, \frac{\left(p_1+p_2\right) \widetilde \lambda_{(n)}-p_1}{\sqrt{n p_1}} \right)$, where $\left( \widetilde \lambda_{1}, \ldots, \widetilde \lambda_{n-1} \right)$ be the $\beta$-Jacobi ensemble $\mathcal J_{n-1}(p_1-1,p_2-1)$.

Similarly, we will also employ the decomposition $$
	\mathrm{~d}Q_{n-1}^{p_1-1, p_2-1}
=  B_{n-1}  \prod_{i=1}^{n-2}(x_{n-1}-x_i)^{\beta} u_n(x_{n-1}) \mathbf 1_{\left\{x_{n-1}> x_{n-2}\right\}} \mathrm{~d}x_{n-2} \mathrm{~d}Q_{n-2}^{p_1-2, p_2-2}, $$ where
$$ B_{n-1} = 	(n-1) \left(\frac{s_{1,n} }{s_{1,n}+s_{2,n}} \right)^{\beta (n-1) +  r_{2, n}}\left(\frac{s_{2,n} }{s_{1,n}+s_{2,n}}\right)^{\beta (n-1) + r_{1, n}} \frac{C_{n-1}^{p_1-1, p_2-1}}{C_{n-2}^{p_1-2, p_2-2}}.$$

Next, we shift our focus to $X_{(n)}.$
\subsection{Proof of Theorem \ref{maxldp}}

When $0<\gamma<1,$ the large deviations of $X_{(n)}$ could be obtained via that of $p \lambda_{(n)}/p_1,$ which has been established in \cite{MaLD}. Since our method is unified for $0\le \gamma<1,$ we state the proof for all $\gamma\in [0, 1).$   
Observe that the function $J_{\gamma,\sigma}$ is continuous and strictly increasing on $\left( \widetilde u_2, \frac{1}{\sqrt{\gamma}\sigma}\right)$ and $\lim\limits_{x \rightarrow \frac{1}{\sqrt{\gamma}\sigma}} J_{\gamma, \sigma}(x)=+\infty,$ where  $ \widetilde{u}_2 $ is defined in \eqref{defu2}. Hence, $J_{\gamma, \sigma}$ is a good rate function. Therefore, as in \cite{Pak}, we are going to check the following three limits
\begin{itemize}
	\item[(1)] for all $ x \in \left( \widetilde u_2, \frac{1}{\sqrt{\gamma}\sigma}\right),$ it holds that 
		\begin{align}\label{main1} 
			\varlimsup_{n \rightarrow \infty} \frac{1}{\beta n} \log \mathbb{P}\left(X_{(n)} \geq x\right) \leq-J_{\gamma, \sigma}(x); 
		\end{align}
 	\item[(2)] for all $ x \in \left( \widetilde u_2, \frac{1}{\sqrt{\gamma}\sigma}\right),$ it is true that 
 		\begin{align}\label{main2}
 			\lim _{\delta \rightarrow 0} \varliminf_{n \rightarrow \infty} \frac{1}{\beta n} \log \mathbb{P}\left(X_{(n)} \in(x-\delta, x+\delta)\right) \geq- J_{\gamma, \sigma}(x);
 		\end{align}
  	\item[(3)] for all $ x < \widetilde u_2,$ we have that 
  		\begin{align}\label{main3}
			\varlimsup_{n \rightarrow \infty} \frac{1}{\beta n} \log \mathbb{P}\left(X_{(n)}\leq x \right) = -\infty
		 \end{align}
\end{itemize}
for weak large deviation and 
\begin{align}\label{main4}
\lim _{M \to  \frac{1}{\sqrt{\gamma}\sigma}  } \varlimsup_{n \rightarrow \infty} \frac{1}{\beta n} \log \mathbb{P}\left(X_{(n)}>M\right)=-\infty
 \end{align}
for the exponential tightness.

We endow the space $\mathcal M_1(\mathbb R)$ of probability measure on $ \mathbb R $ with the $L^1$-Wasserstein distance $d$ which metricizes the weak convergence. 
Before the verifications, for simplicity,  given $\delta>0,$
set
\begin{align*} 
		B(\delta) = \left\{ \mu \in \mathcal M_1(\mathbb R) \, | \,  d(\mu, \widetilde \nu_{\gamma, \sigma})< \delta \right\} 
\end{align*} 
and
\begin{align*} 
		B_{a, b}(\delta) = \left\{ \mu \in B(\delta) \, | \,   \mathrm{supp}( \mu ) \subset [-a, b]  \right\}
\end{align*} 
for $a, \, b>0.$ 

\subsubsection{\bf The verification of \eqref{main3}.} For all $x<  \widetilde u_2 $, it is easy to have
$$
\mathbb{P}\left(X_{(n)} \leq x\right) \leq \mathbb{P}\left(L_n\left(x, \widetilde u_2\right]=0\right),
$$
where $L_n$ is the empirical measure of $\left(X_i\right)_{1 \leq i \leq n}$. Since $\widetilde u_2$ is the righthand point of the support $\widetilde \nu_{\gamma, \sigma}$, there exists a continuous bounded function $f$ supported on $\left(x, \widetilde u_2\right]$ such that $\int f(x) \widetilde \nu_{\gamma, \sigma}(\mathrm d x)>0$ while $\int f(x) L_n(\mathrm d x)=0$. Therefore, there exists a neighborhood $O \in \mathcal{M}_1(\mathbb{R})$ of $\widetilde \nu_{\gamma, \sigma}$ with respect to the weak topology such that
$$
O \cap\left\{L_n: L_n\left(\left(x, \widetilde u_2\right]\right)=0\right\}=\emptyset.
$$
Here $\mathcal{M}_1(E)$ denote the collection of all probabilities with support on $E$ for $E$ being a Borel subset of $\mathbb{R}$. Therefore,
By the large deviation principle for $L_n$, we have
$$
\varlimsup_{n \rightarrow \infty} \frac{1}{\beta n^2} \log \mathbb{P}\left(L_n \in O^c\right) \leq- \inf _{\mu \in O^c} I_{\gamma, \sigma}(\mu)<0
$$
This immediately implies \eqref{main3}.
\subsubsection{\bf The verification of \eqref{main4}.} 

By using the decomposition in \eqref{orderpdf1}, we know
\begin{align*}
			\mathbb{P}\left(X_{(n)}>M\right) =& \int_{\left\{ (x_1,\ldots,x_n) \in \Delta_n | x_n >M \right\} }  B_n \prod_{i=1}^{n-1}(x_n-x_i)^{\beta} u_n(x_n) \mathrm{~d}x_n \mathrm{~d}Q_{n-1}^{p_1-1,p_2-1} \\
 \leq &  B_n  \int_{M}^{s_{2, n}^{-1}} u_n(x_n) \mathrm{~d}x_n \int_{ \Delta_{n-1} }\prod_{i=1}^{n-1}(x_n-x_1)^{\beta} \mathrm{~d}Q_{n-1}^{p_1-1,p_2-1},
\end{align*}
where, in the second step, we use a simple estimate $\prod_{i=1}^{n-1} (x_n-x_i)^{\beta}\leq (x_n-x_1)^{\beta(n-1)}$ for $(x_1,\ldots,x_n) \in \Delta_n$. Here, $$\Delta_n:=\{(x_1, \cdots, x_n)\in\mathbb{R}^n: -s^{-1}_{1,n} \leq x_1\le x_2\le \cdots\le x_n  \leq s^{-1}_{2,n} \}.$$
It is noteworthy that Lemma \ref{interup} leads to the following upper bound
\begin{align*}
	\int_{ \Delta_{n-1} }\prod_{i=1}^{n-1}(x_n-x_1)^{\beta} \mathrm{~d}Q_{n-1}^{p_1-1,p_2-1}
 \le  e^{O(\beta n)} (x_n^{\beta(n-1)}+1) .
\end{align*}
According to Lemma \ref{I01}, there exists a positive constant $a_1 \in (\tilde u_2, \frac{1}{\sqrt{\gamma}\sigma} ) $ such that for all $M \in \left(a_1, \frac{1}{\sqrt{\gamma}\sigma}\right)$
\begin{align*}
	\varlimsup_{n\rightarrow \infty} \frac{1}{\beta n} \log \left( \int_M^{s_{2, n}^{-1}} (x_n^{\beta(n-1)}+1) u_n(x_n) \mathrm{~d} x_n  \right) \leq \varphi_{\gamma, \sigma}(M),
\end{align*}
 where $\varphi_{\gamma, \sigma}$ defined as in \eqref{vargs}.
Consequently, utilizing Lemma \ref{logBn}, we can deduce that
\begin{align*}
	\lim_{M \to \frac{1}{\sqrt{\gamma}\sigma}  } \varlimsup_{n \to \infty} \frac{1}{\beta n} \log \mathbb{P}\left(X_{(n)}>M\right)\le \lim _{M \to  \frac{1}{\sqrt{\gamma}\sigma}  }\varphi_{\gamma, \sigma}(M)+O(1)=-\infty. 
\end{align*}

Next, we employ a similar approach to handle $\mathbb P\left(X_{(1)}<-L\right)$, yielding
\begin{align*}
	 \mathbb P\left(X_{(1)}<-L\right) 
	\leq & B_n\int_{\left\{ (x_1,\ldots,x_n) \in \Delta_n | x_1 <-L \right\} }  u_n(x_1)   \left(x_n-x_1\right)^{\beta(n-1)} \mathrm{~d} x_1\mathrm{~d} Q_{n-1}^{p_1-1,p_2-1}\\
	\leq  & e^{O(\beta n)} \int_{-s_{1, n}^{-1}}^{-L}(\left|x_1\right|^{\beta(n-1)}+1)u_n(x_1) \mathrm{~d} x_1 .
\end{align*}
It follows again from Lemma \ref{I01} that 
\begin{align*}  
		\lim_{L \to \frac{1}{\sqrt{\gamma}}}  \varlimsup_{n\rightarrow \infty} \frac{1 }{\beta n} \log \mathbb P\left( X_{(1)}<-L\right) \leq \lim_{L \to \frac{1}{\sqrt{\gamma}}}\left( \varphi_{\gamma, \sigma}(-L) +O(1)\right) =-\infty.
\end{align*}

 \subsubsection{\bf The verification of \eqref{main1}.}
 Given $\delta>0.$ It's important to observe that  
\begin{align}\label{xneq1}
\mathbb{P}\left(X_{(n)} \geq x\right) \leqslant &  \mathbb{P}\left(X_{(n)} \in[x, M], X_{(1)}> -L, L_{n-1} \in B_{L, M}(\delta) \right) + \mathbb{P}\left(X_{(n)}>M\right)\\ & + \mathbb{P}\left(X_{(n)} \in[x, M], X_{(1)}> -L, L_{n-1} \notin B_{L, M}(\delta) \right)  + \mathbb{P}\left(X_{(1)}<-L\right)\notag
\end{align}
for $ -\frac{1}{\sqrt{\gamma}} < -L < \widetilde{u}_1$ and $\widetilde{u}_2 < x < M < \frac{1}{\sqrt{\gamma}\sigma},$ where $L_{n-1}=\frac{1}{n-1}\sum_{i=1}^{n-1}\delta_{X_{(i)}}.$  
Here, the terms $\mathbb{P}\left(X_{(n)}>M\right)$ and $\mathbb{P}\left(X_{(1)}<-L\right)$ become negligible as $M\to \frac{1}{\sqrt{\gamma} \sigma}$ and $L \to \frac{1}{\sqrt{\gamma}}$. 

By the decomposition \eqref{orderpdf1}, we have that  
\begin{equation}\label{noting}\aligned & \mathbb  P\left(X_{(n)} \in[x, M], X_{(1)}> -L,  L_{n-1} \notin B_{L, M}(\delta)  \right)\\
=& B_n \int_{x}^M u_n(x_n) \mathrm{~d} x_n \int_{[-L, x_n]^{n-1}}\prod_{i=1}^{n-1}(x_n-x_i)^{\beta}\mathbf 1_{\left\{L_{n-1}\notin B_{L, M}(\delta)\right\}}\mathrm{~d} Q_{n-1}^{p_1-1, p_2-1}.\\	
\endaligned 
\end{equation}
On the condition $-L<x_1\le \cdots \le x_n<M$, one gets that 
$$\prod_{i=1}^{n-1}|x_n-x_i|^{\beta}\le (M+L)^{\beta (n-1)}$$ 
and thanks to the elementary inequality $(1+x)\le e^x,$ we have that 
$$
u_n(x_n)\le \exp\left\{ (r_{1, n}-1)s_{1, n}x_n-(r_{2, n}-1)s_{2, n}x_n\right\} = \exp\left\{ O\left( \frac{\beta n\sqrt{np_1}M}{p_2}\right)\right\} .
$$
Thus, we obtain that 
$$\aligned
	& \log \mathbb P\left(X_{(n)} \in[x, M], X_{(1)}> -L,  L_{n-1} \notin B_{L, M}(\delta)  \right) \le B_n e^{ O( \beta n)} Q_{n-1}^{p_1-1, p_2-1} (L_{n-1} \notin B_{L, M}(\delta) ) .  
	\label{xn1not}
\endaligned $$
The large deviation  result for $L_{n-1}$ (Theorem 1.7 in \cite{Ma}) implies that 
$$ 
\varlimsup_{n \rightarrow \infty} \frac{1}{\beta n^2} \log Q_{n-1}^{p_1-1, p_2-1} (L_{n-1} \notin B_{L, M}(\delta) ) < 0.
$$ 
Lemma \ref{logBn} ensures that $\log B_n=O(\beta n).$ Thereby, it holds that  
\begin{align}\label{xneq2}  
	\varlimsup_{n\to \infty} \frac{1}{\beta n} \log \mathbb P\left(X_{(n)} \in[x, M], X_{(1)}> -L,  L_{n-1} \notin B_{L, M}(\delta)  \right) = -\infty   
\end{align}
for any $\delta>0.$
Hence, it remains to explore the asymptotic of the following probability 
$$ 
\mathbb{P}\left(X_{(n)} \in[x, M], X_{(1)}>-L, L_{n-1} \in B_{L, M}(\delta)\right).
$$ Define $$ g(z, \mu):= \int \log |z-y|\left(\mu(\mathrm{d} y)-\widetilde{\nu}_{\gamma, \sigma}(\mathrm{d} y)\right)$$ and $\varphi_{n-1}(x)=\log u_n(x)/(\beta (n-1)).$  We rewrite the target as 
$$\aligned &\mathbb{P}\left(X_{(n)} \in[x, M], X_{(1)}>-L, L_{n-1} \in B_{L, M}(\delta)\right) \\
	=  & B_n \int_x^M \exp \left\{\beta(n-1)\left(\int \log \left|x_n-y\right| \widetilde \nu_{\gamma, \sigma} (\mathrm{d} y)+	\varphi_{n-1}\left(x_n\right)\right)\right\}\mathrm{~d} x_n  \\
		& \times \int_{\left[-L, x_n\right]^{n-1}} \exp \left\{\beta(n-1) g\left(x_n, L_{n-1}\right) \right\} \mathbf 1_{\{L_{n-1} \in B_{L, M}(\delta) \}} \mathrm{d} Q_{n-1}^{p_1-1, p_2-1}. 
		\endaligned $$
		Observe that 
		\begin{align*} 
		& \int \log \left|x_n-y\right| \widetilde{\nu}_{\gamma, \sigma}(\mathrm{d} y)+\varphi_{n-1}\left(x_n\right) \\
	= & \int \log \left|x-y\right| \widetilde{\nu}_{\gamma, \sigma}(\mathrm{d} y)+\varphi_{n-1}\left(x\right)+ \int_x^{x_n}\left(\int \frac{1}{t-y} \widetilde{\nu}_{\gamma, \sigma}(\mathrm{d} y)+\varphi_{n-1}^{\prime}(t)\right) \mathrm{~d} t.
 \end{align*}
It is easy to check that 
$\lim\limits_{n \to \infty}\varphi_{n-1}^{\prime}(t) =\varphi'_{\gamma, \sigma} (t)$ uniformly on $[x, M].$ 
 Lemma \ref{stieltjes} ensures that
\begin{align*}
& \sup _{t \in[x, x_n]} \left(\int \frac{1}{t-y} \widetilde{\nu}_{\gamma, \sigma}(\mathrm{d} y)+\varphi_{\gamma, \sigma}^{\prime}(t) \right) \\ = & \sup_{t\in [x, x_n]}\frac{ \sqrt{\left(\widetilde{u}_2-t\right)\left(\widetilde{u}_1-t\right)}}{(1+\sqrt{\gamma} \,t)(\sigma \sqrt{\gamma} \,t-1)} \left(\gamma \sigma \mathbf 1_{\left\{ 0<\sigma \gamma < 1 \right\}} + \frac{1+\sigma}{2}\mathbf 1_{\left\{ \sigma \gamma =0 \right\}} \right)  <0 
\end{align*} 
since $\widetilde{u}_2<x<x_n<M<\frac{1}{\sigma \,\sqrt{\gamma}}.$ 

% We claim that 
% \begin{align}\label{snu}  
% \sup_{t \in [x, M]} \left(  \int \frac{1}{t-y} \widetilde{\nu}_{\gamma, \sigma}(\mathrm{d} y)+\varphi_{n-1}^{\prime}(t) \mathrm{~d} t \right) <0  
% \end{align} on the condition $ \widetilde u_2 < x<t<M, $ for $n$ large enough. 
 Therefore, 
\begin{align}\label{xneq3} 
		 & \mathbb{P}\left(X_{(n)} \in[x, M], X_{(1)}>-L, L_{n-1} \in B_{L, M}(\delta)\right) \\
 \leq & B_n \exp \left\{\beta(n-1)\left(\int \log|x-y| \widetilde \nu_{\gamma, \sigma}(\mathrm{d} y)+\varphi_{n-1}(x)+\sup_{\mu \in B_{L, M}(\delta)} \sup_{z \in [x,M]} g(z, \mu)\right)\right\} 
		. \nonumber
\end{align}
Once the following limit  \begin{align}\label{sslog} 
\lim_{\delta \downarrow 0} \sup _{\mu \in B_{L, M}(\delta)} \sup _{z \in[x, M]} g(z, \mu) = 0 
\end{align}
is verified, we see clearly from \eqref{xneq3} and the fact $\lim\limits_{n\to\infty}\varphi_{n-1}=\varphi_{\gamma, \sigma}$ uniformly on $[-L, M]$ that 
\begin{equation}\label{xneq5}\aligned & \varlimsup_{n\to\infty}\frac{1}{\beta n}\log \mathbb{P}\left(X_{(n)} \in[x, M], X_{(1)}>-L, L_{n-1} \in B_{L, M}(\delta)\right) \\
\le & z_{\gamma, \sigma}+\int \log |x-y| \widetilde{\nu}_{\gamma, \sigma}(\mathrm{d} y)+\varphi_{\gamma, \sigma}(x). \endaligned \end{equation} 
The Laplace principle, the limits \eqref{xneq2} and \eqref{xneq5} ensure the desired result \eqref{main1}. 

Next, we are going to check first the limit \eqref{sslog}. Since $\widetilde{\nu}_{\gamma, \sigma}$ is in $B_{L, M}(\delta),$ it is clear that 
$$\lim_{\delta \downarrow 0} \sup _{\mu \in B_{L, M}(\delta)} \sup _{z \in[x, M]} g(z, \mu) \geq 0.$$
We only need to prove the opposite direction. Given $ \eta $ a constant satisfying $ 0<\eta < \left( x -\widetilde u_2 \right)\wedge 1.$ For any $\mu \in B_{L, M}(\delta)$ and $z\in [x, M],$ set $$\aligned \Omega_1 :&= \{y \in \mathrm{supp}( \mu ) \cup \mathrm{supp}( \widetilde \nu_{\gamma, \sigma} )\;| \; \,|z-y|> \eta \}; \\
\Omega_2 :&= \{y \in \mathrm{supp}( \mu ) \cup \mathrm{supp}( \widetilde \nu_{\gamma, \sigma} )\; |\; |z-y|\leq \eta \};\endaligned $$ and set $f_z(y):=\log |z-y|$ for $y \in \Omega_1.$ Note that $ \log |z-y| <0 $ for $y\in \Omega_2$ and $\mathrm{supp}( \widetilde \nu_{\gamma, \sigma} ) \subset \Omega_1.$ It holds for any $y_1,\, y_2\in \Omega_1$ that 
$$\left|\log |z-y_1|-\log|z-y_2|\right|=\left|\log\left(1+\frac{|z-y_1|-|z-y_2|}{|z-y_2|}\right)\right|\le \frac{1}{\eta}|y_1-y_2|$$ 
uniformly on $z\in [x, M].$ 
Thus, the Kantorovich and Rubinstein duality of $L_1$-Wasserstein distance implies that 
$$\int_{\Omega_1} \log \left|z-y\right|\left(\mu(\mathrm{d} y)-\widetilde{\nu}_{\gamma, \sigma}(\mathrm{d} y)\right)\le \eta^{-1} d(\mu, \widetilde{\nu}_{\gamma, \sigma})<\eta^{-1}\delta $$ uniformly on $z\in [x, M].$ 
Hence, 
\begin{align*} 
		& \sup _{z \in[x, M]} \int \log |z-y|\left(\mu(\mathrm{d} y)-\widetilde{\nu}_{\gamma, \sigma}(\mathrm{d} y)\right) \\ 
 \leq & \sup_{z \in[x, M]}\int_{\Omega_1} \log \left|z-y\right|\left(\mu(\mathrm{d} y)-\widetilde{\nu}_{\gamma, \sigma}(\mathrm{d} y)\right) +\sup_{z \in[x, M]}\int_{\Omega_2} \log \left|z-y\right|\mu(\mathrm{d} y) \\
 \leq & \eta^{-1} \delta
\end{align*}
for any $\mu \in B_{L, M}(\delta).$ This finishes the proof of \eqref{sslog}. 

The verification of \eqref{main1}
 is finally completed. 

 \subsubsection{The verification of \eqref{main2}.} Let $2\delta < \left(x-\widetilde{u}_{2}\right) \wedge\left(\frac{1}{\sqrt{\gamma} \sigma}-x\right)$ and fix $r \in (\widetilde{u}_{2}, x-2\delta),$ $-L \in (- \frac{1}{\sqrt{\gamma}},  \widetilde{u}_1).$ Similarly, as for \eqref{xneq3}, it holds that  
 \begin{align}\label{Xmax} & \mathbb{P}\left(X_{(n)} \in(x-\delta, x+\delta)\right) \nonumber\\
 \geq & \,\mathbb{P}\left(X_{(n)} \in(x-\delta, x+\delta), L_{n-1} \in B_{L, r}(\delta) \right) \nonumber  \\
 \geq & 2\delta B_n \exp \left\{ \beta(n-1) \inf_{z \in (x- \delta, x+\delta)}\inf_{\mu \in B_{L, r}(\delta)} \int \log |z-y|\mu(\mathrm{d} y) \right\} \nonumber\\
 & \times \exp\left\{  \beta(n-1) \inf_{z \in (x- \delta, x+\delta)} \varphi_{n-1}(z) \right\}\,Q_{n-1}^{p_1-1, p_2-1}(L_{n-1} \in B_{L, r}(\delta)).
  \end{align}
 The large deviation of $L_{n-1}$ implies that $$Q_{n-1}^{p_1-1, p_2-1}(L_{n-1} \in B_{L, r}(\delta))\to 1, \quad n \to \infty.$$  
 Hence, by Lemma \ref{phix} and Lemma \ref{logBn},
 \begin{align*}& \lim _{\delta \downarrow 0} \varliminf_{n \rightarrow \infty} \frac{1}{\beta n} \log \mathbb{P}\left(X_{(n)} \in(x-\delta, x+\delta)\right) \\
 \geq & z_{\gamma, \sigma} + \lim _{\delta \downarrow 0} \inf _{z \in(x-\delta, x+\delta)} \inf _{\mu \in B_{L, r}(\delta)} \int \log |z-y| \mu(\mathrm{d} y)  +\lim _{\delta \downarrow 0} \varliminf_{n \rightarrow \infty} \inf _{z \in(x-\delta, x+\delta)} \varphi_{n-1}(z) \\
 = & z_{\gamma, \sigma}+\int \log |x-y| \widetilde{\nu}_{\gamma, \sigma}(\mathrm{d} y)+\varphi_{\gamma, \sigma}(x) . \end{align*}
 The last equality comes from the continuity of the mapping $$(z, \mu) \longmapsto \int \log |z-y| \mu(\mathrm{d} y)$$ on $[x- \delta, x+\delta] \times B_{L, r}(\delta) $ and the uniform convergence of $\varphi_n(z)$ to $\varphi_{\gamma, \sigma}(z)$ on $[x- \delta, x+\delta].$

%\begin{figure}[h]
%	\caption{$ y_{\gamma, \sigma}(x)$ with different parameters (The two points on the graph are the lefthand and righthand points of the support $\widetilde{\nu}_{\gamma, \sigma}$).}
%\includegraphics[width=\textwidth]{ygs.png}
%\end{figure}

%%%%%%%%%%%%%%%%%%%%%%%%%%%%%%%%%%%%%%%%%%%%%%%%%

The proof of Theorem \ref{maxldp} is completed now.
\subsection{Proof of Theorem  \ref{minldp}}
	 A similar argument works here as for the proof of Theorem \ref{maxldp}. We need to show that the following three limits
	 
	 \begin{itemize}
	\item[(1)] for all $ x \in \left( -\frac{1}{\sqrt{\gamma}}, \widetilde u_1 \right),$ it holds \begin{align}\label{Xmin1}
 \varlimsup_{n \rightarrow \infty} \frac{1}{\beta n} \log \mathbb{P}\left(X_{(1)} \leq x\right) \leq-I_{\gamma, \sigma}(x)
 \end{align}
 	\item[(2)] for all $ x \in \left( -\frac{1}{\sqrt{\gamma}}, \widetilde u_1 \right),$ it holds \begin{align}\label{Xmin2}\lim _{\delta \rightarrow 0} \varliminf_{n \rightarrow \infty} \frac{1}{\beta n} \log \mathbb{P}\left(X_{(1)} \in(x-\delta, x+\delta)\right) \geq- I_{\gamma, \sigma}(x)		
           \end{align}
  	\item[(3)] for all $ x > \widetilde u_1,$ it holds \begin{align*}
 \varlimsup_{n \rightarrow \infty} \frac{1}{\beta n} \log \mathbb{P}\left(X_{(1)}\geq x \right) = -\infty
 \end{align*}
\end{itemize}
for weak large deviation and \begin{align}\label{main5}
\lim _{L \to \frac{1}{\sqrt{\gamma}}  } \varlimsup_{n \rightarrow \infty} \frac{1}{\beta n} \log \mathbb{P}\left(X_{(1)}< -L\right)=-\infty
 \end{align}
for the exponential tightness, which is nothing but \eqref{main5}.	

The third limit is obtained similarly as \eqref{main3}, via the large deviation for the empirical measure of $\left(X_i\right)_{1 \leq i \leq n}.$ 
Similarly, as for \eqref{xneq1}, we choose $-\frac{1}{\sqrt{\gamma}}<-L<x<\widetilde{u}_1$ and $\widetilde{u}_2<M<\frac{1}{\sigma\sqrt{\gamma}}$ to bound  $\mathbb{P}(X_{(1)}\le x)$ as 
 \begin{align*} 
  \mathbb{P}\left(X_{(1)} \leq x\right) 
&\le \mathbb{P}\left(X_{(n)}\ge M \right)+\mathbb{P}\left(X_{(1)}<-L\right)+\mathbb{P}\left(X_{(1)}\in [-L, x],  X_{(n)}<M , L_{n-1} \notin B_{L, M}(\delta) \right)\\
&\quad +\mathbb{P}\left(X_{(1)}\in [-L, x],  X_{(n)}<M , L_{n-1} \in B_{L, M}(\delta) \right).
\end{align*}
As for \eqref{xneq2}, we are able to have that 
$$\lim_{n\to\infty}\frac{1}{\beta n}\log \mathbb{P}\left(X_{(1)}\in [-L, x],  X_{(n)}<M , L_{n-1} \in B_{L, M}(\delta) \right)=-\infty.$$
Hence, the limit \eqref{main4} and \eqref{main5} and the Laplace principle working together to lead that 
\begin{align*} 
&  \varlimsup_{n \rightarrow \infty} \frac{1}{\beta n} \log \mathbb{P}\left(X_{(1)} \leq x\right)  \\
\le & \lim _{\delta \rightarrow 0}  \varlimsup_{n \rightarrow \infty} \frac{1}{\beta n} \log \mathbb{P}\left(X_{(1)}\in [-L, x],  X_{(n)}<M , L_{n-1} \in B_{L, M}(\delta) \right)  \end{align*}
The same analysis as for \eqref{xneq3} entails that 
\begin{align*} 
& \mathbb{P}\left(X_{(1)}\in [-L, x],  X_{(n)}<M , L_{n-1} \in B_{L, M}(\delta) \right)\\
\leq & B_n\exp \left\{\beta(n-1)\left(\int \log|x-y| \widetilde \nu_{\gamma, \sigma}(\mathrm{d} y)+\varphi_{n-1}(x)+\sup_{\mu \in B_{L, M}(\delta)} \sup_{z \in [-L, x]} g(z, \mu)\right)\right\}. 
\end{align*}
Since $\lim_{n\to\infty}\varphi_{n-1}(x)=\varphi_{\gamma, \sigma}(x)$ uniformly on $[-L, \widetilde{u}_1]$ and similarly $$\lim_{\delta\downarrow 0}\sup_{\mu \in B_{L, M}(\delta)} \sup_{z \in [-L, x]} g(z, \mu)=0,$$ we finally get that 
$$\varlimsup_{n\to\infty}\frac{1}{\beta n}\log \mathbb{P}\left(X_{(1)}\le x \right)\le z_{\gamma, \sigma}+\int \log|x-y| \widetilde \nu_{\gamma, \sigma}(\mathrm{d} y)+\varphi_{\gamma, \sigma}(x).$$
The  limit \eqref{Xmin1} is achieved. 

Next, we work on the last limit \eqref{Xmin2}.
Let $2 \delta<\left(\widetilde{u}_1-x\right) \wedge\left(x+ \frac{1}{\sqrt{\gamma}}\right)$ and fix $- r \in\left(x+2 \delta, \widetilde{u}_1\right)$, $M \in\left( \widetilde{u}_2, \frac{1}{\sqrt{\gamma}\sigma}\right)$. Similar to \eqref{Xmax}, we have
 \begin{align*} & \mathbb{P}\left(X_{(1)} \in(x-\delta, x+\delta)\right) \nonumber\\
  \geq & 2\delta B_n \exp \left\{ \beta(n-1) \inf_{z \in (x- \delta, x+\delta)}\inf_{\mu \in B_{r, M}(\delta)} \int \log |z-y|\mu(\mathrm{d} y) \right\} \nonumber\\
 & \times \exp\left\{  \beta(n-1) \inf_{z \in (x- \delta, \, x+\delta)} \varphi_{n-1}(z) \right\}\,Q_{n-1}^{p_1-1, p_2-1}(L_{n-1} \in B_{r, M}(\delta)).
  \end{align*}
  Therefore, due to the fact $\lim\limits_{n\to\infty}Q_{n-1}^{p_1-1, p_2-1}(L_{n-1} \in B_{r, M}(\delta))=1,$ it is true that 
\begin{align*} 
& \lim _{\delta \downarrow 0} \varliminf_{n \rightarrow \infty} \frac{1}{\beta n} \log \mathbb{P}\left(X_{(1)} \in(x-\delta, x+\delta)\right) \\
\geq & z_{\gamma, \sigma}+\lim _{\delta \downarrow 0} \inf _{z \in(x-\delta, x+\delta)} \inf _{\mu \in B_{r, M}(\delta)} \int \log |z-y| \mu(\mathrm{d} y)+\lim _{\delta \downarrow 0} \varliminf_{n \rightarrow \infty} \inf _{z \in(x-\delta, x+\delta)} \varphi_{n-1}(z)\\
=& z_{\gamma, \sigma}+\int \log |x-y| \widetilde{\mu}_{\gamma, \sigma}(\mathrm{d} y)+\varphi_{\gamma, \sigma}(x),
 \end{align*}
where the last equality is due to the continuity of $$(z, \mu) \longmapsto \int \log |z-y| \mu(\mathrm{d} y)$$ on $[x- \delta, x+\delta] \times B_{r, \,M}(\delta) $ and the fact that $\varphi_n$ converges to $\varphi_{\gamma, \sigma}(x)$ uniformly on $[x- \delta, x+\delta]. $  This closes the whole proof of Theorem \ref{minldp}.

\section{Proofs of Theorems \ref{fnasy} and \ref{thfnasy}}
In this section, we present the proofs of Theorems \ref{fnasy} and \ref{thfnasy}. 

For ease of notations, we use $Q$ and $R$ to denote the  $Q_n^{p_1, p_2}$ and $ R_n^{p_1, p_2},$ respectively. 
Recall the definition of $F_n$ as expressed in \eqref{deff1}.
\subsection{Proof of Theorem \ref{fnasy}} 
In what follows, we use $\mathbb E$ and $\mathbb E^R$ to represent the expectations under the measures $Q$ and $R,$ respectively. Write
$$ \mathbb{E}^R\left[F_n^2\right]  =\mathbb{E}\left[\frac{d Q}{d R}; x<X_{(n)} \right]. $$ Given $-\frac{1}{\sqrt{\gamma}}<-L$ and $\widetilde{u}_2<x<M<\frac{1}{\sqrt{\gamma} \sigma}$ and  $\delta>0,$ 
Recall  $$B_{L, M}(\delta)=\{\mu \in \mathcal{M}_1(R) \mid d\left(\mu, \widetilde{\nu}_{\gamma, \sigma}\right)<\delta, \,\operatorname{supp}(\mu) \subset[-L, M]\}.$$ We still split the target $ \mathbb{E}^R\left[F_n^2\right]$ into four parts as follows 
 \begin{equation}\label{decom}\aligned 
	\mathbb{E}^R\left[F_n^2\right]   
	&=\mathbb{E}\left[\frac{\mathrm{d} Q}{\mathrm{d} R}; x<X_{(n)}<M , X_{(1)}>-L, L_{n-1} \in B_{L, M}(\delta)\right]  \\
	&\quad  +\mathbb{E}\left[\frac{\mathrm{d} Q}{\mathrm{d} R}; X_{(n)}>M \right]	+\mathbb{E}\left[\frac{\mathrm{d} Q}{\mathrm{d} R}; X_{(1)}<-L\right] \\
	&\quad 	+  \mathbb{E}\left[\frac{\mathrm{d} Q}{\mathrm{d} R}; x<X_{(n)}<M , X_{(1)}>-L, L_{n-1} \notin B_{L, M}(\delta)\right]. 
	\endaligned 
 \end{equation} 
We first show that the last three terms in \eqref{decom} are negligible when consider the asymptotic of $\frac1{\beta n}\log \mathbb{E}^R\left[F_n^2\right].$ Precisely, we are going to check that 
\begin{align}
	& \lim _{M \to \frac{1}{\sqrt{\gamma} \sigma}} \varlimsup_{n \rightarrow \infty} \frac{1}{\beta n} \log \mathbb{E}\left[\frac{\mathrm{d} Q}{\mathrm{d} R}; X_{(n)}>M \right] = - \infty;
	\label{qrxm} \\
	& \lim _{L \to \frac{1}{\sqrt{\gamma}}} \varlimsup_{n \rightarrow \infty} \frac{1}{\beta n}\log \mathbb{E}\left[\frac{\mathrm{d} Q}{\mathrm{d} R}; X_{(1)}<-L\right]  = - \infty;
	\label{qrxl} \\
	&  \varlimsup_{n \rightarrow \infty} \frac{1}{\beta n} \log \mathbb{E}\left[\frac{\mathrm{d} Q}{\mathrm{d} R} ; x<X_{(n)}<M, X_{(1)}>-L, L_{n-1} \notin B_{L, M}(\delta)\right] =  - \infty. 
	\label{qrnotb}
\end{align}   
The verifications of the three limits above will be similar to those in the precedent section and this is the reason that we skip some details in the following procedure. 

In fact, using \eqref{deff1} to arrive at 
$$
	\mathbb{E}\left[\frac{\mathrm{d} Q}{\mathrm{d} R}; X_{(n)}>M \right] =B_n\mathbb{E}\left[\prod_{i=1}^{n-1}(X_{(n)}-X_{(i)})^{\beta}u_n(X_{(n)}) h^{-1}(X_{(n)})\mathbf 1_{\left\{X_{(n)}>M\vee X_{(n-1)}\right\}} \right].$$
	Hence, it follows from the decomposition \eqref{orderpdf1} with $m=n$ that 
	\begin{align*}
	\mathbb{E}\left[\frac{\mathrm{d} Q}{\mathrm{d} R}; X_{(n)}>M \right]
	&=B_n^2 \int_{\left\{ (x_1,\ldots,x_n) \in \Delta_n | x_n >M \right\} } \frac{u_n^2(x_n)}{h(x_n)} \prod_{i=1}^{n-1}\left(x_n-x_i\right)^{2\beta} \mathrm{~d} Q_{n-1}^{p_1-1, p_2-1} \\
	&\leq  B_n^2 r^{-1} e^{-rx}  \int_M^{s_{2, n}^{-1}}u_n^2(x_n) e^{rx_n}  \mathrm{~d} x_n \int_{\Delta_{n-1}} \prod_{i=1}^{n-1}\left(x_n-x_1\right)^{2\beta} \mathrm{d} Q_{n-1}^{p_1-1, p_2-1}.
\end{align*}  
 Here,  for the second inequality, we used the elementary inequality $$\prod_{i=1}^{n-1}\left(x_n-x_i\right)^\beta \leq\left(x_n-x_1\right)^{\beta(n-1)}$$ when $x_1<\ldots<x_n$ and the fact that $$ h^{-1}(x_n) = r^{-1} e^{r(x_n- x \vee x_{n-1})} \leq   r^{-1} e^{r(x_n- x)}.$$ The inequality \eqref{upxn1} tells that $$\int_{\Delta_{n-1}} \prod_{i=1}^{n-1}\left(x_n-x_1\right)^{2\beta} \mathrm{d} Q_{n-1}^{p_1-1, p_2-1} \leq e^{O(\beta n)} x_n^{2\beta(n-1)}$$ and the asymptotic of $B_n $ in \eqref{bnlog}  says $B_n = e^{O(\beta n)}.$ 
Thereby, we see that 
$$\mathbb{E}\left[\frac{\mathrm{d} Q}{\mathrm{d} R}; X_{(n)}>M \right]
\leq  e^{O(\beta n)} \int_M^{s_{2, n}^{-1}} x_n^{2\beta(n-1)} u_n^2(x_n) e^{rx_n}  \mathrm{~d} x_n.$$
Lemma \ref{I01} in the Appendix ensures, with $l_n=2\beta(n-1), c_n=r$ and $q=2,$ that 
$$\varlimsup_{n \rightarrow \infty} \frac{1}{\beta n} \log \int_M^{s_{2, n}^{-1}} x_n^{2\beta(n-1)} u_n^2(x_n) e^{rx_n}  \mathrm{~d} x_n \le 2\varphi_{\gamma, \sigma}(M)+ O(1).$$
Consequently, it holds that 
\begin{align*}
	 \lim _{M \to \frac{1}{\sqrt{\gamma} \sigma}} \varlimsup_{n \rightarrow \infty} \frac{1}{\beta n} \log \mathbb{E}\left[\frac{\mathrm{d} Q}{\mathrm{d} R}; X_{(n)}>M \right]  
	 \leq  2\lim _{M \to \frac{1}{\sqrt{\gamma} \sigma}}\varphi_{\gamma, \sigma}(M)+O(1) 
	 =-\infty,
\end{align*}  
where $\varphi_{\gamma, \sigma}$ is defined as in \eqref{vargs}. The limit \eqref{qrxm} follows.

For any $\delta > 0,$ since on the condition $X_{(n)}\in [x, M]$ and $X_{(1)}>-L$, $|X_{(n)}-X_{(i)}|^{\beta}\le (M+L)^{\beta},$
we have the following inequality
\begin{align*}
	& \mathbb{E}\left[\frac{\mathrm{d} Q}{\mathrm{d} R} ; x\le X_{(n)}\le M, X_{(1)}>-L, L_{n-1} \notin B_{L, M}(\delta) \right] \\
	\leq & B_n (M+L)^{\beta(n-1)}\mathbb{E}\left[h^{-1}(X_{(n)})u_n(X_{(n)}); x<X_{(n)}<M, X_{(1)}>-L, L_{n-1} \notin B_{L, M}(\delta) \right]\\
	\leq & B_n^2 (M+L)^{2\beta(n-1)} \int_{x }^M r^{-1} e^{r\left(x_n-x\right)} u^2_n(x_n) \mathrm{~d} x_n \times Q_{n-1}^{p-1, p_2-1}(L_{n-1} \notin B_{L, M}(\delta)) \\
	\leq & e^{O(\beta n)} Q_{n-1}^{p_1-1, p_2-1}\left(L_{n-1} \notin B_{L, M}(\delta)\right).
\end{align*}
The large deviation result for $L_{n-1}$ tells that 
$$
\varlimsup_{n \rightarrow \infty} \frac{1}{\beta n^2} \log Q_{n-1}^{p_1-1, p_2-1}\left(L_{n-1} \notin B_{L, M}(\delta)\right)<0,
$$
which implies that 
$$\varlimsup_{n \rightarrow \infty} \frac{1}{\beta n} \log\mathbb{E}\left[\frac{\mathrm{d} Q}{\mathrm{d} R} ; x<X_{(n)}<M, X_{(1)}>-L, L_{n-1} \notin B_{L, M}(\delta) \right]=-\infty.$$ 

Now we check the second limit \eqref{qrxl}. Indeed, using \eqref{deff1} and \eqref{decom}, we get that 

$$\aligned \mathbb{E}\left[\frac{\mathrm{d} Q}{\mathrm{d} R}; X_{(1)}<-L\right]&=B_n\mathbb{E}\left[\prod_{i=1}^{n-1}(X_{(n)}-X_{(i)})^{\beta}u_n(X_{(n)})h^{-1}(X_{(n)}) \mathbf 1_{\left\{X_{(n)}>x\vee X_{(n-1)}, \; X_{(1)}<-L\right\}}\right] \\
&\le B_n^2 B_{n-1}\int_{-s_{1, n}^{-1}}^{-L} u_n(x_1) \mathrm{~d} x_1 \int_{x}^{s_{2, n}^{-1}} u_n^2(x_n) h^{-1}(x_n)(x_n-x_1)^{2\beta (n-1)} \mathrm{~d} x_n\\
&\le e^{O(\beta n)}\int_{-s_{1, n}^{-1}}^{-L} u_n(x_1) \mathrm{~d} x_1 \int_{x}^{s_{2, n}^{-1}} u_n^2(x_n) (x_n-x_1)^{2\beta (n-1)} \mathrm{~d} x_n.
\endaligned $$
Using the C-R inequality $$(x_n-x_1)^{2\beta(n-1)} \leq 2^{2\beta(n-1)}(|x_n|^{2\beta(n-1)} + |x_1|^{2\beta(n-1)}),$$ we separate 
the integral above as 
$$\aligned &\int_{-s_{1, n}^{-1}}^{-L} u_n(x_1) \mathrm{~d} x_1 \int_{x}^{s_{2, n}^{-1}} u_n^2(x_n) (x_n-x_1)^{2\beta (n-1)} \mathrm{~d}x_n \\
&\le 2^{2\beta(n-1)} \bigg(\int_{-s_{1, n}^{-1}}^{-L} u_n(x_1)|x_1|^{2\beta (n-1)} \mathrm{~d} x_1 \int_{x}^{s_{2, n}^{-1}} u_n^2(x_n)  \mathrm{~d} x_n\\
&\quad+\int_{-s_{1, n}^{-1}}^{-L} u_n(x_1) \mathrm{~d} x_1 \int_{x}^{s_{2, n}^{-1}} u_n^2(x_n) |x_n|^{2\beta (n-1)} \mathrm{~d} x_n\bigg).
\endaligned $$
Eventually, Lemma \ref{I01} helps us to get that 
$$\aligned &\varlimsup_{n\to\infty}\frac{1}{\beta n}\log\mathbb{E}\left[\frac{\mathrm{d} Q}{\mathrm{d} R}; X_{(1)}<-L\right]\le  \varphi_{\gamma, \sigma}(-L)+O(1).
\endaligned $$
Thereby, by the fact $\lim\limits_{L\to \frac{1}{\sqrt{\gamma}}}\varphi_{\gamma, \sigma}(-L)=+\infty, $ we confirm \eqref{qrxl}.

By the Laplace principle and the decomposition \eqref{decom},  after the achievement of \eqref{qrxm}, \eqref{qrxl} and \eqref{qrnotb}, we only need to focus on 
$$ \mathbb{E}\left[\frac{\mathrm{d} Q}{\mathrm{d} R} ; x<X_{(n)}<M, X_{(1)}>-L, L_{n-1} \in B_{L, M}(\delta) \right]. $$
Similarly, as for \eqref{xneq3}, we have
\begin{align}\label{lasta}
	& \mathbb{E}\left[\frac{\mathrm{d} Q}{\mathrm{d} R} ; x<X_{(n)}<M, X_{(1)}>-L, L_{n-1} \in B_{L, M}(\delta) \right] \nonumber\\
	\leq & B_n^2  \int_{x }^M r^{-1} e^{r\left(x_n-x\right)} \exp \left\{2 \beta(n-1) \left( \int \log|x_n-y| \widetilde{\nu}_{\gamma, \sigma}(\mathrm{d} y)   + \varphi_{n-1}\left(x_n\right)\right)\right\} \mathrm{~d} x_n \nonumber\\
	& \times \int_{[-L, M]^{n-1}}  \exp \left\{2 \beta(n-1) g\left(x_n, L_{n-1}\right)\right\} \mathbf{1}_{\left\{L_{n-1} \in B_{L, M}(\delta)\right\}} \mathrm{~d} Q_{n-1}^{p_1-1, p_2-1} \nonumber \\
	\leq & B_n^2 \exp \left\{2 \beta(n-1)\left(\int \log |x-y| \widetilde{\nu}_{\gamma, \sigma}(\mathrm{d} y)+\varphi_{n-1}(x)\right)\right\} \nonumber\\
	& \times \int_x^M \exp \left\{2 \beta(n-1) \int_x^{x_n}\left(\frac{r}{2\beta(n-1)} + \int \frac{1}{t-y} \widetilde{\nu}_{\gamma, \sigma}(\mathrm{d} y)+\varphi_{n-1}^{\prime}(t)\right) \mathrm{~d} t\right\} \mathrm{~d} x_n \nonumber\\
	& \times \exp \left\{2 \beta(n-1) \sup _{\mu \in B_{L, M}(\delta)} \sup _{z \in[x, M]} g(z, \mu)\right\},
\end{align}  
where $g(z, \mu):=\int \log |z-y|\left(\mu(\mathrm{d} y)-\widetilde{\nu}_{\gamma, \sigma}(\mathrm{d} y)\right).$ The limit \eqref{sslog} can set free the last term from the product on the right hand side of \eqref{lasta}.
Recall that $$\int \frac{1}{t-y} \widetilde{\nu}_{\gamma, \sigma}(\mathrm{d} y)+\varphi_{n-1}^{\prime}(t) $$ uniformly converges to $-y_{\gamma, \sigma}^{\prime}(t)$ on $[x, M],$ and note that $J_{\gamma, \sigma}'=y_{\gamma, \sigma}'$ on $[x, M],$ which is increasing on $[x, M].$ 
It follows, 
under the condition $ r < 2 \beta (n-1)J_{\gamma, \sigma}^{\prime}(x), $ that for $n$ large enough
\begin{align*}&\int_x^M \exp \left\{2 \beta(n-1) \int_x^{x_n}\left(\frac{r}{2\beta(n-1)} + \int \frac{1}{t-y} \widetilde{\nu}_{\gamma, \sigma}(\mathrm{d} y)+\varphi_{n-1}^{\prime}(t)\right) \mathrm{~d} t\right\} \mathrm{~d} x_n
\\
\le &\int_x^M \exp \left\{2 \beta(n-1)\int_x^{x_n} (J'_{\gamma, \sigma}(x)-J_{\gamma, \sigma}'(t)) dt\right\} dx_n	\\
\le & M-x.
\end{align*}
Putting these facts into \eqref{lasta}, with the help of Lemmas \ref{logBn} and \ref{phix}, one gets that 
\begin{align*}
	& \lim_{\delta \downarrow 0} \varlimsup_{n \rightarrow \infty} \frac{1}{\beta n} \log \mathbb{E}\left[\frac{\mathrm{d} Q}{\mathrm{d} R} ; x<X_{(n)}<M, X_{(1)}>-L, L_{n-1} \in B_{L, M}(\delta)\right] \\
	\leq & 2z_{\gamma, \sigma}+2\int \log |x-y| \widetilde{\nu}_{\gamma, \sigma}(\mathrm{d} y)+2\varphi_{\gamma, \sigma}(x) \\
	=& -2J_{\gamma, \sigma}(x).
\end{align*}  
Hence, together with the large deviation on $X_{(n)},$ it holds that 
$$
\varlimsup _{n \rightarrow \infty} \frac{\log \mathbb{E}^R\left[F_n^2\right] }{2 \log \mathbb P\left(X_{(n)}> x\right)} \leq 1 .
$$
On the other hand, we know 
$$ 
	\mathbb{E}^R\left[F_n^2\right]
	\geq\left( \mathbb E^R\left[ \frac{\mathrm{d} Q}{\mathrm{d} R} ;  X_{(n)}>x \right]\right)^2
	= \mathbb P\left( X_{(n)}>x \right)^2
$$ 
by Jessen's inequality. The two facts imply the desired conclusion.

\subsection{Proof of Theorem \ref{thfnasy}}
Recall 
$G_n  = \frac{\mathrm d Q_n^{p_1, p_2}}{\mathrm d T_n^{p_1, p_2}}\left(X_{(1)}, \cdots, X_{(n)}\right) \mathbf 1_{\{X_{(1)} < x\} } $ with $-\frac{1}{\sqrt{\gamma}}<x<\widetilde{u}_1.$  For simplicity, we use $Q$ and $T$ to denote   $Q_n^{p_1, p_2}$ and $ T_n^{p_1, p_2},$ respectively. Therefore, the importance sampling estimator $G_n = \frac{\mathrm{d} Q}{\mathrm{d} T} \left(X_{(1)}, \cdots, X_{(n)}\right)\mathbf 1_{\left\{X_{(1)}<  x\right\}} $ can be written as
\begin{equation}\label{defG1} \aligned 
	G_n & = \frac{g_{n}^{p_1, p_2}\left(X_{(1)}, \cdots, X_{(n)}\right)}{g_{n-1}^{p_1-1, p_2-1}\left(X_{(2)}, \cdots, X_{(n)}\right) \hat{h}(X_{(1)})}\mathbf 1_{\left\{X_{(1)}<x\right\}}\\
	& = B_n \prod_{i=2}^{n}(X_{(i)} -X_{(1)})^{\beta} u_n (X_{(1)})\frac{\mathbf 1_{\left\{X_{(1)} < x \wedge X_{(2)} \right\}}}{\hat{h}(X_{(1)})}.   
	\endaligned 
 \end{equation}
 To get the asymptotical efficiency of $G_n,$ i.e., the following limit
$$
\lim _{n \rightarrow \infty} \frac{\log \mathbb{E}^T\left[G_n^2\right] }{2 \log \mathbb P\left(X_{(1)}< x\right)} = 1,
$$
as stated for $F_n,$ we only need to prove that 
\begin{align}\label{suplim}
	\varlimsup_{n\to\infty}\frac{1}{\beta n}\log \mathbb{E}\left[\frac{dQ}{dT}; X_{(1)}<x\right]\le 2z_{\gamma, \sigma}+2\int \log |x-y| \widetilde{\nu}_{\gamma, \sigma}(\mathrm{d} y)+2\varphi_{\gamma, \sigma}(x).
\end{align}
The proof of \eqref{suplim} will be completely similar to the proof related to $F_n$ and then we will skip many details. 
In fact, using \eqref{defG1} to claim that for any $A\in \mathcal{B}(\mathbb{R}^n),$ 
$$\aligned 
	\mathbb{E}\left[\frac{\mathrm{d} Q}{\mathrm{d} T}\, \mathbf 1_{A}\right]& =B_n\mathbb{E}\left[\prod_{i=2}^{n}(X_{(i)}-X_{(1)})^{\beta}u_n(X_{(1)})\hat{h}^{-1}(X_{(1)})\mathbf 1_{A}\right]\\
	&\le r^{-1} e^{r x}B_n\, \mathbb{E}\left[(X_{(n)}-X_{(1)})^{\beta(n-1)}u_n(X_{(1)}) e^{-r X_{(1)}}\mathbf 1_{A}\right].
	\endaligned $$
	We see clearly that in the expression of $\mathbb{E}\left[\frac{\mathrm{d} Q}{\mathrm{d} T}\, \mathbf 1_{\left\{X_{(1)}<x\right\}}\right],$ $X_{(1)}$ and $X_{(n)}$ interchange in that of  
	Given any $-\frac{1}{\sqrt{\gamma}}<-L<x<\widetilde{u}_1$ and 
	$\widetilde{u}_2<M<\frac1{\sigma\sqrt{\gamma}},$ similar arguments as for \eqref{qrxm} and \eqref{qrxl} will lead 
	$$\lim_{L\to \frac{1}{\sqrt{\gamma}}}\varlimsup_{n\to\infty}\frac1{\beta n}\log \mathbb{E}\left[\frac{\mathrm{d} Q}{\mathrm{d} T}; X_{(1)}<-L  \right]=-\infty$$
	and 
	$$\lim_{M\to \frac{1}{\sigma\sqrt{\gamma}}}\varlimsup_{n\to\infty}\frac1{\beta n}\log \mathbb{E}\left[\frac{\mathrm{d} Q}{\mathrm{d} T}; X_{(n)}>M  \right]=-\infty,$$
	respectively. 
	Given $\delta>0,$ taking $$A_1=\left\{x_1\le \cdots\le  x_n: -L\le x_{1}\le x, \, x_{n}\le M, \, L_{n-1}=\frac{1}{n-1}\sum_{i=2}^n \delta_{x_i}\notin B_{L, M}(\delta)\right\},$$
	we know that 
	$$\mathbb{E}\left[\frac{\mathrm{d} Q}{\mathrm{d} T}\,  \mathbf 1_{A_1}\right]\le r^{-1} e^{r (x+L)}B_n |M+L|^{\beta (n-1)} \mathbb{E}[u_n(X_{(1)})\mathbf 1_{A_1}].$$
	Hence, it follows from the decomposition \eqref{orderpdf1} that 
	\begin{align*}
	 \mathbb{E}\left[u_n(X_{(1)})\mathbf 1_{A_1}\right]
	=& B_n \int_{-L}^{x} u_n^2(x_1)\mathrm{~d} x_1 \int  \prod_{i=2}^{n}\left(x_i-x_1\right)^{\beta} \mathbf 1_{A_1}\mathrm{~d} Q_{n-1}^{p_1-1, p_2-1}(x_2, \cdots, x_n) \\
	\leq & B_n |M+L|^{\beta (n-1)} \int_{-L}^{x} u_n^2(x_1) \mathrm{~d} x_1 Q_{n-1}^{p_1-1, p_2-1}(L_{n-1}\notin B_{L, M}(\delta)).
\end{align*}  
The large deviation result for $L_{n-1}$ tells again that 
$$
\varlimsup_{n \rightarrow \infty} \frac{1}{\beta n} \log Q_{n-1}^{p_1-1, p_2-1}\left(L_{n-1} \notin B_{L, M}(\delta)\right)=-\infty,
$$
which implies that 
$$\varlimsup_{n \rightarrow \infty} \frac{1}{\beta n} \log\mathbb{E}\left[\frac{\mathrm{d} Q}{\mathrm{d} T} \, \mathbf 1_{A_1} \right]=-\infty.$$ 
Set 
$$A_2=\left\{x_1\le \cdots\le  x_n: -L\le x_{1}\le x, \, x_{n}\le M, \, L_{n-1}=\frac{1}{n-1}\sum_{i=2}^n \delta_{x_i}\in B_{L, M}(\delta)\right\}.$$
Since on the condition $x_1\le x_2\le \cdots \le x_n,$ we see that $$\{x_{1} \le x\}\subset \{x_{1}<-L\}\cup \{x_{n}>M\}\cup A_1 \cup A_2.$$
Hence, the Laplace principle guarantees that for \eqref{suplim} it suffices to prove that 
\begin{equation}\label{tar1}\aligned \varlimsup_{n\to\infty}\frac{1}{\beta n}\log \mathbb{E}\left[\frac{\mathrm{d} Q}{\mathrm{d} T}\mathbf 1_{A_2}\right]\le 2z_{\gamma, \sigma}+2\int \log |x-y| \widetilde{\nu}_{\gamma, \sigma}(\mathrm{d} y)+2\varphi_{\gamma, \sigma}(x). 
\endaligned \end{equation} 
In fact,   as for \eqref{lasta}, we have  that  
\begin{align*} & \mathbb{E}\left[\frac{\mathrm{d} Q}{\mathrm{d} T}\, \mathbf 1_{A_2}\right]\\
=& B_n ^2\int_{-L}^{x} u_n^2(x_1) \hat{h}^{-1} (x_1)\mathrm{~d} x_1 \int  \prod_{i=2}^{n}\left(x_i-x_1\right)^{2\beta} \mathbf 1_{A_2}\mathrm{~d} Q_{n-1}^{p_1-1, p_2-1}(x_2, \cdots, x_n) \\
	\leq & B_n^2 \exp \left\{2 \beta(n-1)\left(\int \log |x-y| \widetilde{\nu}_{\gamma, \sigma}(\mathrm{d} y)+\varphi_{n-1}(x)\right)\right\} \nonumber\\
	& \times \int_{-L}^x \exp \left\{2 \beta(n-1) \int_{-L}^{x_1} \left(\frac{r}{2\beta(n-1)} + \int \frac{1}{t-y} \widetilde{\nu}_{\gamma, \sigma}(\mathrm{d} y)+\varphi_{n-1}^{\prime}(t)\right) \mathrm{~d} t\right\} \mathrm{~d} x_1.\nonumber\end{align*}  

It follows, 
under the condition $\frac{r}{2 \beta (n-1)} < I_{\gamma, \sigma}^{\prime}(x)  $ and the fact $$\int \frac{1}{t-y} \widetilde{\nu}_{\gamma, \sigma}(\mathrm{d} y)+\varphi_{n-1}^{\prime}(t) \to -I^{\prime}_{\gamma, \sigma}(t)\quad \text{uniformly on } [-L, x],$$ that for $n$ large enough
\begin{align*}&\int_{-L}^x \exp \left\{2 \beta(n-1) \int_{-L}^{x_1}\left(\frac{r}{2\beta(n-1)} + \int \frac{1}{t-y} \widetilde{\nu}_{\gamma, \sigma}(\mathrm{d} y)+\varphi_{n-1}^{\prime}(t)\right) \mathrm{~d} t\right\} \mathrm{~d} x_1
\\
\le &\int_{-L}^x \exp \left\{2 \beta(n-1)\int_{-L}^{x_1} (I'_{\gamma, \sigma}(x)-I_{\gamma, \sigma}'(t)) dt\right\} dx_1	\\
\le & x+L.
\end{align*}	
Here,  for the last inequality we use the property of $I_{\gamma, \sigma}$ that $I'_{\gamma, \sigma}$ is decreasing on $[-L, x].$ 
So, we have that 
\begin{align*}\mathbb{E}\left[\frac{\mathrm{d} Q}{\mathrm{d} T}\, \mathbf 1_{\left\{A_2\right\}}\right]
	\leq B_n^2 \exp \left\{2 \beta(n-1)\left(\int \log |x-y| \widetilde{\nu}_{\gamma, \sigma}(\mathrm{d} y)+\varphi_{n-1}(x)\right)\right\}, \end{align*}  
which, together with Lemmas \ref{phix} and \ref{logBn}, implies that 
\begin{align*} \varlimsup_{n\to\infty}\frac{1}{\beta n}\log \mathbb{E}\left[\frac{\mathrm{d} Q}{\mathrm{d} T}\, \mathbf 1_{A_2}\right]\le 2z_{\gamma, \sigma}+2\int \log |x-y| \widetilde{\nu}_{\gamma, \sigma}(\mathrm{d} y)+2\varphi_{\gamma, \sigma}(x) .\nonumber\end{align*}  
This is the target \eqref{tar1}  and the proof is then completed.

\section{Appendix}
We present first a foundational lemma that will be recurrently used in this paper. Recall the definitions of $s_{i, n}=\sqrt{p_1 n} / p_i$ and $r_{i, n}= \beta\left(p_i-n+1\right)/2$ for $i=1, 2,$ and $u_n(x)=\left(1+s_{1, n} x\right)^{r_{1, n}-1}\left(1-s_{2, n} x\right)^{r_{2, n}-1}.$ 
\begin{lem}\label{phix}
For $x \in \left( - \frac{1}{\sqrt{\gamma}}, \frac{1}{\sqrt{\gamma} \sigma} \right)$ and any constant $a$, define
 For $x \in \left( - \frac{1}{\sqrt{\gamma}}, \frac{1}{\sqrt{\gamma} \sigma} \right)$ and any constant $a$,	define 
$$ \varphi_{n-a} (x) =\frac{r_{1,n}-a}{\beta (n-a)}\log (1+s_{1,n}x)+ \frac{r_{2,n}-a}{\beta (n-a)}\log  (1-s_{2,n}x).$$ 
Under the assumption {\bf A}, it holds that
$$ \lim_{n\rightarrow \infty} \varphi_{n-a} (x) = \varphi_{\gamma, \sigma}(x)$$ 
uniformly on $[-L, M],$
where $\varphi_{\gamma, \sigma}$ defined as in \eqref{vargs}.
\end{lem}

Next, we give a lemma that plays a crucial role in the proofs of Theorems \ref{fnasy} and \ref{thfnasy}.  

\begin{lem}\label{I01}

Suppose that the assumption  {\bf A} holds.  Given $d>0.$ For any sequences $\ell_n$ and $c_n$ satisfying $\lim\limits_{n \rightarrow \infty} \frac{\ell_n}{\beta n} = \ell$ and $\varlimsup_{n \rightarrow \infty} \frac{c_n}{\beta n} = c,$ define \begin{align*}
	I_1(x):= \int_{x}^{s_{2, n}^{-1}}  t^{\ell_n} \left(u_n(t)\right)^de^{c_n t}\mathrm{~d}t 
	\quad
	\text{and}
	\quad
	I_2(x):= \int_{x}^{s_{1, n}^{-1}}  t^{\ell_n }u_n(t) e^{c_n t} \mathrm{~d}t.
\end{align*}
We have the following asserstions. 
\begin{itemize}
	\item [(i)] If $ 0 \leq \gamma \sigma <1$ and $1 -2 \sqrt{\gamma}c \mathbf 1_{\{ \sigma =0 \}}>0,$  there exists a positive constant $a_1 \in (\tilde u_2, \frac{1}{\sqrt{\gamma}\sigma} ) $ such that
$$
	\varlimsup_{n\rightarrow \infty} \frac{\log I_1(x)}{\beta n } \leq  \ell \log x + c x + d \varphi_{\gamma, \sigma}(x), \text{ for all }x \in ( a_1, \frac{1}{\sqrt{\gamma}\sigma}).
$$
	\item[(ii)]
If $0 \leq \gamma <1$, there exists a positive constant $a_2$ such that
$$
	\varlimsup_{n\rightarrow \infty} \frac{\log I_2(x)}{\beta n } \leq  \ell \log x + c x + \varphi_{\gamma, \sigma}(-x), \text{ for all } x \in (a_2, \frac{1}{\sqrt{\gamma}}).	
$$ 
\end{itemize}
Here, the function $\varphi_{\gamma, \sigma}$ is defined in \eqref{vargs}.

\end{lem}
\begin{proof}
Let $$\omega_n(x):= -\ell_n \log x - d(r_{1,n}-1) \log(1+s_{1,n}x)-d(r_{2,n}-1)\log(1-s_{2,n}x)-c_n x.$$
Under the condition {\bf A}, a direct calculation gives that
\begin{align*}
	\frac{\omega_n^{\prime} (x) }{\beta n}
	 & \longrightarrow 
	 - \frac{\ell+cx}{x} + \frac{d(1+\sigma-2\gamma\sigma)x+d\sqrt{\gamma}(1-\sigma)}{2\left(1+\sqrt{\gamma} x\right)\left(1-\sqrt{\gamma} \sigma x\right)}=: \kappa_{\gamma, \sigma}(x)
\end{align*}
as $n \to \infty.$
Set $ n_1 = 2  \vee  \frac{4 \ell \gamma \sigma +2c \sqrt{\gamma}+1- \gamma }{1-\gamma \sigma}$ and
\begin{align*}
	a_1 
	& = 
	\begin{cases}
		\frac{n_1-1}{n_1}\frac{1}{\sqrt{\gamma}\sigma}, & 0<\sigma \gamma <1; \\
		1 \vee \left(\sqrt{\gamma}(2\ell +2 c-1)+1\right), & \sigma=0,0<\gamma \leq 1;\\
		1 \vee \frac{2(\ell+c)}{1+\sigma}, & \gamma=0, \sigma \geq 0.
	\end{cases}
\end{align*}  
 When $0<\sigma \gamma < 1,$ set $ n_1 = 2  \vee  \frac{4 \ell \gamma \sigma +2c \sqrt{\gamma}+1- \gamma }{1-\gamma \sigma}$ and $a_1= \frac{n_1-1}{n_1}\frac{1}{\sqrt{\gamma}\sigma}.$ Then,
\begin{align*}
	\kappa_{\gamma, \sigma}\left(a_1 \right)
	& \geq -\left( 2 \ell  \sqrt{\gamma}\sigma +c\right) - \frac{1-\gamma}{2\sqrt{\gamma}}+ \frac{n_1(1-\gamma\sigma)}{2\sqrt{\gamma}} \geq 0.
\end{align*}  
%When $\sigma=0,0<\gamma \leq 1,$ set $a_1 =  1 \vee \left(\sqrt{\gamma}(2\ell +2 c-1)+1\right).$ Thus,
%\begin{align*}
%	\kappa_{\gamma, \sigma}\left(a_1 \right)
%	& \geq
%	-\left(\ell +c\right) - \frac{1+\sqrt \gamma}{2\sqrt{\gamma}}+ \frac{1}{2\sqrt{\gamma}}x \geq 0.
%\end{align*}  
%When $\gamma=0, \sigma \geq 0,$ set $a_1= 1 \vee \frac{2(\ell+c)}{1+\sigma} .$ Thus,
%\begin{align*}
%	\kappa_{\gamma, \sigma}\left(a_1 \right)
%	& \geq
%	 -\left(\ell +c\right) + \frac{1+\sigma}{2}x \geq0 .
%\end{align*}  
 Using the sign-preserving property, we know $\omega_n'(a_1)/(\beta n)>0$ for $n$ large enough.
Observing  the form of $\omega_n,$ we see that 
the function $\omega_n$ is a strictly convex function on $\left(0,  s_{2,n}^{-1}\right)$. Thus, $$ 
	\omega_n^{\prime} (x) = O(\beta n) \quad \text{and}\quad  \omega_n^{\prime} (x) > 0 \text{ for all }x \in  \left(a_1, \frac{1}{\sqrt{\gamma}\sigma}\right)
$$ 
for $n$ large enough. 
Thereby,  for any $x \in  \left(a_1, \frac{1}{\sqrt{\gamma}\sigma}\right)$,
\begin{align*}
	\log I_1(x) & \leq \log  \left(\int_{x}^{s_{2, n}^{-1}} e^{-\omega_n(t)} \omega_n^{\prime}(t) \mathrm{~d}t\right) - \log \left(\inf_{t \in [x, s_{2,n}^{-1}] } \omega_n^{\prime}(t)\right) \\
	& = -  \omega_n (x)-\log(\omega_n'(x)). 
\end{align*}  
Thus, 
$$\varlimsup_{n\rightarrow \infty} \frac{\log I_1(x)}{\beta n } \leq  \varlimsup_{n\rightarrow \infty} \frac{-\omega_n(x)}{\beta n }.$$ 
Taking $a=1$ in Lemma \ref{phix} and using the conditions $\lim\limits_{n \rightarrow \infty} \frac{\ell_n}{\beta n} = \ell$ and $\varlimsup_{n \rightarrow \infty} \frac{c_n}{\beta n} = c,$ , we get $$\varlimsup_{n\rightarrow \infty} \frac{\log I_1(x)}{\beta n } \leq  \ell \log x + c x + d \varphi_{\gamma, \sigma}(x) $$ for all $x\in (a_1, )$
 
 The proof of (ii) will be similar.  Let $$\widetilde \omega_{n}(x):= -\ell_n \log x - (r_{1,n}-1) \log(1-s_{1,n}x)-(r_{2,n}-1)\log(1+s_{2,n}x)-c_n x.$$ Then,
 \begin{align*}
 	\frac{\widetilde \omega_n^{\prime}(x)}{\beta n} \longrightarrow
 		- \frac{(\ell+cx)\left( (1+\sigma-2 \gamma \sigma) x- \sqrt{\gamma}(1+\sigma)\right)}{2x(1-\sqrt{\gamma} x)(1+\sqrt{\gamma} \sigma x)}=: \widetilde \kappa_{\gamma, \sigma}(x).
 \end{align*}
Set $n_2 = 2 \vee  \frac{4 \ell \gamma +2c \sqrt{\gamma}+1- \gamma \sigma }{1-\gamma}$ and
\begin{align*}
	a_2 = 
	\begin{cases}
		 \frac{n_2-1}{n_2}\frac{1}{\sqrt{\gamma}} , & 0 < \gamma < 1, 0 \leq \sigma \gamma \leq 1;\\  
		1 \vee \frac{2(\ell+c)}{1+\sigma} , & \gamma=0, \sigma \geq0.
	\end{cases}
\end{align*}  
Similar arguments tell $ \widetilde \kappa_{\gamma, \sigma}(a_2) > 0$ and  
\begin{align*}
	\widetilde \omega_n^{\prime}(x)
	=
	O(\beta n) \quad \text{and} \quad  \widetilde \omega_n^{\prime}(x)>0 \quad   \text { for all } x \in\left(a_2, \frac{1}{\sqrt{\gamma}}\right)
\end{align*}  
 for $n$ large enough. 
Thereby,  for any $x \in  \left(a_2, \frac{1}{\sqrt{\gamma}}\right)$,
\begin{align*}
	\log I_2(x) & \leq \log  \left(\int_{x}^{s_{1,n}^{-1}} e^{-\widetilde \omega_n(t)} \widetilde \omega_n^{\prime}(t) \mathrm{~d}t\right) - \log \left(\inf_{t \in [x, s_{1,n}^{-1}] } \widetilde \omega_n^{\prime}(t)\right) \\
	& = -  \widetilde \omega_n (x)-\log(\widetilde \omega_n^{\prime}(x)). 
\end{align*}  
Applying Lemma \ref{phix} with $a=1$,  we get the conclusion of (ii). 
\end{proof}

Next, we state a lemma on the normalizing constant $B_n$ defined in \eqref{defBn}.
\begin{lem}\label{logBn}
$B_n$ is defined as in \eqref{defBn}, it holds
\begin{align}\label{bnlog}
	\log B_n  = - \frac{\beta n}{2} \log \frac{p_2}{p} +\frac{\beta(p-n)}{2} \log \frac{p-1}{p-n}+ O( \beta \log n ). 
\end{align}
Under the assumption {\bf A}, we have
\begin{align*}  
	z_{\gamma,\sigma}:=
	&  \lim_{n \to \infty} \frac{\log B_n}{\beta n} = \frac{1}{2} \log(1+\sigma )+   \frac{1+\sigma-\gamma\sigma}{2\gamma\sigma} \log \left( 1+ \frac{\gamma\sigma}{1+\sigma-\gamma\sigma}\right). 
\end{align*}
Here $z_{0,\sigma} = \lim\limits_{\gamma \to 0}z_{\gamma,\sigma} =   \frac{1}{2} \log(1+\sigma ) + \frac{1}{2} $ and $z_{\gamma, 0} = \lim\limits_{\sigma \to 0}z_{\gamma,\sigma} = \frac{1}{2}. $
\end{lem}

\begin{proof}
	By definition,
$$
 \frac{C_n^{p_1, p_2}}{C_{n-1}^{p_1-1, p_2-1}}=\frac{\Gamma\left(1+\frac{\beta}{2}\right) \Gamma\left(\frac{\beta p}{2}\right) \Gamma\left(\frac{\beta(p-1)}{2}\right)}{\Gamma\left(1+\frac{\beta n}{2}\right) \Gamma\left(\frac{\beta p_1}{2}\right) \Gamma\left(\frac{\beta p_2}{2}\right) \Gamma\left(\frac{\beta(p-n)}{2}\right)}.
$$ Thus,
\begin{align*} 
	\log  \frac{C_n^{p_1, p_2}}{C_{n-1}^{p_1-1, p_2-1}}
	= &\log \Gamma\left(1+\frac{\beta}{2}\right)+\log \Gamma\left(\frac{\beta p}{2}\right)+\log \Gamma\left(\frac{\beta(p-1)}{2}\right) \\
	& -\log \Gamma\left(1+\frac{\beta n}{2}\right)-\log \Gamma\left(\frac{\beta p_1}{2}\right)-\log \Gamma\left(\frac{\beta p_2}{2}\right)-\log \Gamma\left(\frac{\beta(p-n)}{2}\right).
\end{align*}
Recall the Stirling formula $$
\log \Gamma(x)=\left(x-\frac{1}{2}\right) \log x-x+O(1)
$$
for x large enough. It implies that
\begin{align*} 
	\log  \frac{C_n^{p_1, p_2}}{C_{n-1}^{p_1-1, p_2 -1}}= 
	&\log \Gamma\left(1+\frac{\beta}{2}\right)-\frac{\beta}{2} \log \frac{\beta}{2}+  \frac{\beta p_1}{2} \log \frac{p}{p_1} +\frac{\beta p_2}{2} \log \frac{p}{p_2} \\ 
	& \quad +  \left(\frac{\beta (p-n)}{2}-\frac{1}{2}\right) \log \frac{p-1}{p-n} +\frac{\beta(n-1)}{2} \log \frac{p}{n} \\
	& + \frac 12 \log \frac{p_1}{2} - \frac{\beta+1}{2}\log \frac{n}{2} +O(1).
\end{align*}
Since $\log \Gamma\left(1+\frac{\beta}{2}\right)-\frac{\beta}{2} \log \frac{\beta}{2}=o(\beta)$ for $\beta$ large enough and is bounded for finite $\beta$, 
\begin{align*}
	\log B_n 
	= &  \frac{\beta(n-1)+1}{2}\log \frac{n}{p_1} + \log \frac{p_1}{p} \left(\frac{\beta(p_1+n-1)}{2} - \frac{\beta p_1}{2} \right) \\
	& + \log \frac{p_2}{p} \left(\frac{\beta(p_2-n+1)}{2} - \frac{\beta p_2}{2} \right) +\left(\frac{\beta(p-n)}{2}-\frac{1}{2}\right) \log \frac{p-1}{p-n} \\ 
	& +\frac{\beta(n-1)}{2} \l og \frac{p}{n}+\frac{1}{2} \log \frac{p_1}{2} + O( \beta \log n )  \\
	= & - \frac{\beta n}{2} \log \frac{p_2}{p} +\frac{\beta(p-n)}{2} \log \frac{p-1}{p-n}+ O( \beta \log n ).
\end{align*}
The conclusion is obtained by taking the limit.
\end{proof}

The following Lemma is important when one deals with the tail probability like 
$\mathbb{P}(X_{(n)}>x).$

\begin{lem}\label{interup} Let $(\lambda_1, \cdots, \lambda_{n-1})$ be the $\beta$-Jabobi ensembles and $Q_{n-1}^{p_1-1, p_2-1}$ be the joint distribution of $(X_{(1)}, \cdots, X_{(n-1)}).$ Given any $x_n\in (-s_{1,n}^{-1}, s_{2,n}^{-1}),$ we have 
\begin{align}\label{upxn1}
	\int_{ \Delta_{n-1}}\prod_{i=1}^{n-1}(x_n-x_1)^{\beta} \mathbf 1_{\left\{x_{n}>x_{n-1}\right\}}\mathrm{~d}Q_{n-1}^{p_1-1,p_2-1} 
	\leq e^{O(\beta n)} (|x_n|^{\beta(n-1)}+1).
 \end{align}
 Here $\Delta_{n-1}:=\{(x_1, \cdots, x_{n-1})\in\mathbb{R}^n: -s^{-1}_{1,n} \leq x_1\le x_2\le \cdots\le x_{n-1}  \leq s^{-1}_{2,n} \}.$

 \end{lem}

\begin{proof}
The C-R inequality tells us that  $$(x_n-x_1)^{\beta(n-1)} \leq 2^{\beta(n-1)}(|x_n|^{\beta(n-1)} + |x_1|^{\beta(n-1)}).$$
Hence, 
\begin{align*} 
	\int_{\Delta_{n-1} }\prod_{i=1}^{n-1}(x_n-x_i)^{\beta} \mathbf 1_{\left\{x_{n}>x_{n-1}\right\}}\mathrm{~d}Q_{n-1}^{p_1-1,p_2-1}  
	 \leq  2^{\beta(n-1)}
	 \left( |x_n|^{\beta(n-1)} + \int_{\Delta_{n-1} }|x_1|^{\beta(n-1)} \mathrm{~d}Q_{n-1}^{p_1-1,p_2-1} \right).
 \end{align*}
Using the decomposition $$\mathrm{d}Q_{n-1}^{p_1-1, p_2-1}
		=  B_{n-1} \prod_{i=2}^{n-1}(x_i-x_1)^{\beta} u_{n}(x_1) \mathbf 1_{\left\{x_2> x_{1}\right\}} \mathrm{~d}x_1 \mathrm{~d}Q_{n-2}^{p_1-2, p_2-2}
$$ and the inequality $0<x_i-x_1\le x_n-x_1,$ and the C-R inequality again, 
we have that 
\begin{align*}
	\int_{\Omega_{n-1}}|x_1|^{\beta(n-1)} \mathrm{~d}Q_{n-1}^{p_1-1,p_2-1}
	\leq  &B_{n-1}   \int_{-s_{1,n}^{-1}}^{s_{2,n}^{-1}}  |x_1|^{\beta(n-1)} \left(x_n-x_1\right)^{\beta(n-1)}  u_n(x_1)\mathrm{~d}x_1  \\
	\leq  & 2^{\beta(n-1)}B_{n-1} \int_{-s_{1,n}^{-1}}^{s_{2,n}^{-1}} \left(|x_n|^{\beta(n-1)}+ |x_1|^{\beta(n-1)} \right) |x_1|^{\beta(n-1)} u_n(x_1) \mathrm{~d}x_1. 
 \end{align*}
By Lemma \ref{I01}, for any $\ell_n = O(\beta n)$, we can find bounded constants $a_0$ and $a_1$ such that   
 \begin{align*} 
	\int_{-s_{1,n}^{-1}}^{s_{2,n}^{-1}}  |x_1|^{\ell_n} u_n(x_1) \mathrm{~d}x_1 
	&=\left(\int_{-s_{1,n}^{-1}}^{-a_0}+\int_{-a_0}^{a_1}+ \int_{a_1}^{s_{2,n}^{-1}} \right)  |x_1|^{\ell_n} u_n(x_1) \mathrm{~d}x_1\\
	&=:I_1+I_2+I_3.\end{align*} 
	Lemma \ref{I01} guarantees that $I_1+I_3=e^{O(\beta n)}$ and so is $I_2$ since the integrated function is continuous with order $e^{O(\beta n)}$ and the interval is bounded.
  Using the approximation of $B_n$ in \eqref{bnlog}, we finally accomplish the proof of \eqref{upxn1}. 	
\end{proof}

%%%%%%%%%%%%%%%%%%%%%%

Stieltjes transforms (also called Cauchy transforms in the literature) of functions of bounded variation are another important tool in random matrix theory. If $G$ is a function of bounded variation on the real line, then its Stieltjes transform is defined by
$$
S_G(z)=\int \frac{1}{\lambda-z}  \mathrm{~d} G(\lambda), \quad z \in \{z \in \mathbb{C}: \Im z > 0\}.
$$

\begin{lem}\label{stieltjes} Recall the density function of $\widetilde{\nu}_{\gamma, \sigma}$ is defined in \eqref{dennu}. 
Let $$ I(z):= \mathbf 1_{D}(z)-\mathbf 1_{E}(z),$$ where  $D:= \{\Im z >0 \}  \cup \{ \Im z = 0, \Re z >\widetilde u_2\}$ and $E:= \{\Im z <0 \}  \cup \{ \Im z = 0, \Re z <\widetilde u_1\}.$ The Stieltjes transform $\widetilde{S}_{\gamma, \sigma}$ of the probability  $\widetilde{\nu}_{\gamma, \sigma}$ satisfies for any $z\in\mathbb{C},$ 
\begin{align*} 
	\widetilde{S}_{\gamma, \sigma}(z)  
	=	\begin{cases}
				\frac{\sqrt{\gamma}(1-\sigma)-z(1+\sigma-2\gamma\sigma)}{2(1+\sqrt{\gamma}z)(1-\sqrt{\gamma}\sigma z )} +\frac{\gamma \sigma \sqrt{ (\widetilde u_2 -z)(\widetilde u_1-z)}}{(1+\sqrt{\gamma} z)(1-\sigma \sqrt{\gamma} z)} \; I(z), 
				& 0<\sigma \gamma \leq 1;\\ 
				\frac{\sqrt{\gamma}-z}{2(1+\sqrt{\gamma}z)}+ \frac{\sqrt{(\gamma-z)^2-4}}{2(1+\sqrt{\gamma} z)}\; I(z), 
				& \sigma=0,0<\gamma \leq 1; \\ 
				-\frac{(1+\sigma) z}{2}+\frac{1+\sigma}{2} \sqrt{z^2-\frac{4}{1+\sigma}}\; I(z) ,   
				& \gamma=0, \sigma \geq 0.
 	 \end{cases}  
 \end{align*} 
 \end{lem}
\begin{proof}
	We first state the case $\Im z > 0.$ From \cite{DuEd} and \cite{Gam},  we know that the Stieltjes transforms of the semicircular law, the Marchenko-Pastur law and the Wachter law are as follows:
\begin{align*}
	S_{c_2}(z)= \int_{-2}^{2} \frac{1}{x-z} c_2(x) \mathrm{~d}x= - \frac{z}{2}+\frac{\sqrt{z^2-4}}{2}
\end{align*} and 
\begin{align*}
	S_{\gamma}(z)=\int_{\gamma_1}^{\gamma_2} \frac{1}{x-z} h_r(x) \mathrm{~d}x=\frac{1-\gamma-z}{2 \gamma z}+\frac{\sqrt{(1+\gamma-z)^2-4 \gamma}}{2 \gamma z} 
\end{align*} 
and 
\begin{align*}
	S_{\gamma \sigma}(z) 
	=\int_{u_1}^{u_2} \frac{1}{x-z} h_{\gamma, \sigma}(x) \mathrm{~d}x 
	=\frac{1-\gamma}{2 \gamma z }-\frac{1-\gamma \sigma}{2 \gamma \sigma} \frac{1}{1-z}+\frac{1+\sigma}{\sigma} \frac{\sqrt{\left(u_2-z\right)(u_1-z)}}{z(1-z)} .
\end{align*}

 For other cases, we can use the residue theorem to evaluate the integral. Next, we evaluate the conclusion in semicircular. Letting $z=2  \cos y$ and $\zeta=e^{i y}$, then
$$S_{c_2}(z)=-\frac{1}{4 i \pi} \oint_{|\zeta|=1} \frac{\left(\zeta^2-1\right)^2}{\zeta^2\left(\zeta^2-z \zeta +1\right)} \mathrm{~d}\zeta.$$ Note that the integrand has three poles, at $\zeta_0=0, \zeta_1= \frac{z+\sqrt{z^2-4 }}{2} $, and $\zeta_2=\frac{z-\sqrt{z^2-4 }}{2} .$ By simple calculation, we find that the residues at these three poles are
$$
z \text { and } \sqrt{z^2-4 } \text { and } =-\sqrt{z^2-4 }.
$$
It is clear that
$$\begin{cases}
\left|\zeta_0\right|<1,\left|\zeta_1\right|>1,\left|\zeta_2\right|<1, & z \in  D_1;\\
\left|\zeta_0\right|<1,\left|\zeta_1\right|< 1,\left|\zeta_2\right|>1, & z \in  E_1;\\
\left|\zeta_0\right|<1,\left|\zeta_1\right|\leq 1,\left|\zeta_2\right|\leq1, & z \in  \mathbb C \setminus  D_1 \setminus  E_1 ;\\
\end{cases} $$
where $ D_1 = \{\Im z>0\} \cup\left\{\Im z=0, \Re z > 2\right\} $ and $ E_1 = \{\Im z<0\} \cup\left\{\Im z=0, \Re z < -2\right\}, \; \forall z \in \mathbb C.$
Therefore by Cauchy integration, we obtain 
\begin{align*}
	S_{c_2}(z) 
	= \int_{-2}^{2} \frac{1}{x-z} c_2(x) \mathrm{~d}x 
	= - \frac{z}{2}+\frac{\sqrt{z^2-4}}{2}\left(\mathbf{1}_{D_1}(z)-\mathbf{1}_{E_1}(z)\right).
\end{align*}
Similarly,  a direct calculation gives that the integral corresponding to Marchenko-Pastur law is
\begin{align*}
	S_{\gamma}(z)
	=\frac{1-\gamma-z}{2 \gamma z}+\frac{\sqrt{(1+\gamma-z)^2-4 \gamma}}{2 \gamma z} \left(\mathbf{1}_{D_2}(z)-\mathbf{1}_{E_2}(z)\right),
\end{align*}
where $ D_2 = \{\Im z>0\} \cup\left\{\Im z=0, \Re z > \gamma_2\right\} $ and $ E_2 = \{\Im z<0\} \cup\left\{\Im z=0, \Re z < \gamma_1 \right\}.$ The integral corresponding to Wachter law is
\begin{align*}
	S_{\gamma \sigma}(z)
	=\frac{1-\gamma}{2 \gamma z }-\frac{1-\gamma \sigma}{2 \gamma \sigma} \frac{1}{1-z}+\frac{1+\sigma}{\sigma} \frac{\sqrt{\left(u_2-z\right)(u_1-z)}}{z(1-z)}\left(\mathbf{1}_{D_3}(z)-\mathbf{1}_{E_3}(z)\right)
\end{align*} 
where $ D_3 = \{\Im z>0\} \cup\left\{\Im z=0, \Re z > u_2 \right\} $ and $ E_3 = \{\Im z<0\} \cup\left\{\Im z=0, \Re z < u_1 \right\}.$
Thus, we have
\begin{align*}
	\widetilde{S}_{\gamma, \sigma}(z) 
	=
		\begin{cases}
			\frac{\sigma \sqrt{\gamma}}{1+\sigma} S_{\gamma, \sigma}\left(\frac{\sigma}{1+\sigma}(1+\sqrt{\gamma} z)\right), 
			& 0<\sigma \gamma \leq 1; \\
			\sqrt{\gamma} S_\gamma(1+\sqrt{\gamma} z), 
			& \sigma=0,0<\gamma \leq 1; \\ 
			\sqrt{1+\sigma}S_{c_2}(\sqrt{1+\sigma}z ),
			& \gamma=0, \sigma \geq 0. 
		\end{cases}	
\end{align*}

Then the proof follows from basic calculations.
\end{proof}

\end{document}